\documentclass[a4paper,11pt]{amsart}
\usepackage[utf8x]{inputenc}

\usepackage{amsmath}
\usepackage{amsfonts}
\usepackage{amssymb}

\newtheorem{thm}{Theorem}[subsection]
\newtheorem*{thmszn}{Theorem} 

\newtheorem{cor}{Corollary}[subsection]

\newtheorem{rem}{Remark}[subsection]

\newtheorem{Opp}{Open Problem}

\numberwithin{equation}{subsection}

\title[Some functional equations ...]{Some functional equations related to the characterizations of information measures and their stability}

\author[E.~Gselmann]{Eszter Gselmann}
\address{
Department of Analysis\\
Institute of Mathematics\\
University of Debrecen\\
P. O. Box: 12.\\
Debrecen\\
H--4010\\
Hungary}

\email{gselmann@science.unideb.hu}

\author[Gy.~Maksa]{Gyula Maksa}
\address{
Department of Analysis\\
Institute of Mathematics\\
University of Debrecen\\
P. O. Box: 12.\\
Debrecen\\
H--4010\\
Hungary}

\email{maksa@science.unideb.hu}

\thanks{This research has been supported by the Hungarian Scientific Research Fund (OTKA)
Grant NK 814 02 and by the T\'{A}MOP 4.2.1./B-09/1/KONV-2010-0007 project implemented
through the New Hungary Development Plan co-financed by the European Social Fund and
the European Regional Development Fund.}



\date{\today}

\begin{document}
\maketitle

\begin{abstract}
The main  purpose of this paper is to investigate the stability problem of
some functional equations that appear in the characterization
problem of information measures.
\end{abstract}


\section{Introduction and preliminaries}

Throughout this paper
$\mathbb{N}, \mathbb{Z}, \mathbb{Q}, \mathbb{R}$ and $\mathbb{C}$
will stand for the set of the positive integers, the integers,
the rational numbers, the reals and the set of the complex numbers, respectively.
Furthermore,
$\mathbb{R}_{+}$ and $\mathbb{R}_{++}$ will denote the
set of the nonnegative and the positive real numbers, respectively.

In this section, firstly we summarize some notations and preliminaries
that will be used subsequently.
We begin with the introduction of the information measures. Here their definition and
some results concerning them will follow.

The second section of our paper will be devoted to the
topic of information functions. Here -- among others -- the general solution of the
(parametric) fundamental equation of information will be described.
Furthermore, some results concerning the so-called sum form information measures will also be listed.

Finally, in the last part of this paper, we will investigate the stability problem
for the functional equations that appeared in the second section.
Here some open problems will also be presented.

\subsection{Information measures}

The question 'How information can be measured?', was firstly raised by Hartley in 1928.
In his paper \cite{Har28}, Hartley considered only those systems of events, in which
every event occurs with the same probability.
After that, in 1948 the celebrated paper of Shannon \cite{Sha48} appeared where the
information quantity contained in a complete (discrete) probability distribution was defined.

In what follows, based on the notions and the results of the monograph
Acz\'{e}l--Dar\'{o}czy \cite{AczDar75}, a short introduction to
information measures will follow.

Let $n\in\mathbb{N}$, $n\geq 2$ be arbitrarily fixed and define the sets
\[
\Gamma^{\circ}_{n}=
\left\{
(p_{1}, \ldots, p_{n})\in\mathbb{R}^{n} \vert p_{i}>0, i=1, \ldots, n, \sum^{n}_{i=1}p_{i}=1
\right\}
\]
and
\[
\Gamma_{n}=
\left\{
(p_{1}, \ldots, p_{n})\in\mathbb{R}^{n} \vert
p_{i}\geq 0, i=1, \ldots, n, \sum^{n}_{i=1}p_{i}=1
\right\},
\]
respectively. We say that the sequence of functions $\left(I_{n}\right)_{n=2}^{\infty}$ (or simply $\left(I_{n}\right)$) is an
\emph{information measure}, if either
$I_{n}\colon \Gamma^{\circ}_{n}\to\mathbb{N}$ for all $n\geq 2$ or
$I_{n}\colon \Gamma_{n}\to\mathbb{N}$ for all $n\geq 2.$

We have to mention that, in the literature, 'information measures' depending on not only
probabilities but on the events themselves (inset information measures) (see e.g Acz\'el--Dar\'oczy \cite{AczDar78})
or depending on several probability distributions (see Ebanks--Sahoo--Sander \cite{EbaSahSan97}) are also investigated. Here
we do not involve these cases. On the other hand, originally the zero probabilities were allowed adopting the conventions

\begin{equation}\label{Eq1.1.1}
0\log_{2}(0)=\frac{0}{0+0}=0^{\alpha}=0 \qquad (\alpha\in\mathbb{R})
\end{equation}
in the formulas. In this paper, we follow these conventions and we denote $\Gamma^{\circ}_{n}$ or $\Gamma_{n}$  by $\mathcal{G}_{n}$
provided that it does not matter that the zero probabilities are excluded or not.

Certainly, the most common information measures are the
\emph{Shannon entropy} (see Shannon \cite{Sha48}), i.e.,
\[
H^{1}_{n}(p_{1}, \ldots, p_{n})=-\sum^{n}_{i=1}p_{i}\log_{2}(p_{i}),
\quad
\left((p_{1}, \ldots, p_{n})\in\mathcal{G}_{n}, n\geq 2\right)
\]
and the so-called \emph{entropy of degree $\alpha$}, or the
\emph{Havrda--Charv\'{a}t entropy} (see Acz\'{e}l--Dar\'{o}czy \cite{AczDar75},
Dar\'{o}czy \cite{Dar70}, Kullback \cite{Kul59}, Tsallis \cite{Tsa88}), i.e.,
\[
H^{\alpha}_{n}(p_{1}, \ldots, p_{n})=\left\{
\begin{array}{lcl}
\left(2^{1-\alpha}-1\right)^{-1}\left(\sum^{n}_{i=1}p_{i}^{\alpha}-1\right), &\text{if}& \alpha\neq 1 \\
H^{1}_{n}(p_{1}, \ldots, p_{n}), &\text{if}& \alpha=1
\end{array}
\right.,
\]
where $n\in\mathbb{N},\,\, n\geq 2, \,\, \alpha\in\mathbb{R}\,\,$ and
$(p_{1}, \ldots, p_{n})\in\mathcal{G}_{n}$.

$(H_{n}^{1})$ was first introduced to the statistical
thermodynamics by Boltzmann and Gipps, to the information theory by Shannon \cite{Sha48},
while $(H_{n}^{\alpha})$\,\,(for $\alpha\neq 1$) was first investigated from cybernetic point of view
by Havrda and Charv\'at \cite{HavCha67}, from information theoretical point of view by Dar\'oczy \cite{Dar70},
and was rediscovered by Tsallis \cite{Tsa88} for the Physics community.

It is easy to see that, for arbitrarily fixed $n\geq 2$ and
$(p_{1}, \ldots, p_{n})\in\mathcal{G}_{n}$,
\[
\lim_{\alpha\rightarrow 1}
H^{\alpha}_{n}\left(p_{1}, \ldots, p_{n}\right)=
H^{1}_{n}\left(p_{1}, \ldots, p_{n}\right),
\]
which shows that the Shannon entropy can continuously be embedded to the
family of entropies of degree $\alpha$. As it is formulated in \cite{AczDar75}, the \emph{characterization problem}
for the information measure $(H^{\alpha}_{n})$ is the following: What properties have to be imposed upon an
information measure $(I_{n})$ in order that $(I_{n})=(H^{\alpha}_{n})$ be valid.

In what follows, we list the properties which seem to be reasonable for characterizing $(H^{\alpha}_{n})$. It is not difficult to check
that the information measure $(H_{n}^{\alpha})$ has these properties.

An information measure  $\left(I_{n}\right)$ is called
\emph{symmetric} if
\begin{equation}\label{Eq1.1.2}
I_{n}\left(p_{1}, \ldots, p_{n}\right)=
I_{n}\left(p_{\sigma(1)}, \ldots, p_{\sigma(n)}\right)
\end{equation}
is satisfied for all  $n\geq 2$, $(p_{1}, \ldots, p_{n})\in \mathcal{G}_{n}$ and
for arbitrary permutation $\sigma:\left\{1, \ldots, n\right\}\rightarrow\left\{1, \ldots, n\right\}.$
Further, we say that  $(I_{n})$ is \emph{3-semi-symmetric} if
\begin{equation}\label{Eq1.1.3}
I_{3}(p_{1}, p_{2}, p_{3})=I_{3}(p_{1}, p_{3}, p_{2})
\end{equation}
holds for all $(p_{1}, p_{2}, p_{3})\in\mathcal{G}_{3}$.

$(I_{n})$ is  called \emph{normalized} if
\begin{equation}\label{Eq1.1.4}
I_{2}\left(\frac{1}{2}, \frac{1}{2}\right)=1,
\end{equation}
and it is called \emph{$\alpha$-recursive} if
\begin{multline}\label{Eq1.1.5}
I_{n}\left(p_{1}, \ldots, p_{n}\right) \\
=I_{n-1}\left(p_{1}+p_{2}, p_{3}, \ldots, p_{n}\right)+
\left(p_{1}+p_{2}\right)^{\alpha}I_{2}\left(\frac{p_{1}}{p_{1}+p_{2}}, \frac{p_{2}}{p_{1}+p_{2}}\right)
\end{multline}
holds for all for all $n\geq 3$ and $(p_{1}, \ldots, p_{n})\in \mathcal{G}_{n}$.
In case $\alpha =1$, we say simply that $(I_{n})$ is \emph{recursive}.

For a fixed $\alpha\in\mathbb{R}$ and $2\leq n\in\mathbb{N}, 2\leq m\in\mathbb{N}$, the information measure $(I_{n})$ is said to be
\emph{($\alpha, n,m)$- additive}, if

\begin{equation}\label{Eq1.1.6}
I_{nm}\left(P\ast Q\right)=I_{n}\left(P\right)+I_{m}\left(Q\right)+(2^{1-\alpha}-1)I_{n}\left(P\right)I_{m}\left(Q\right)
\end{equation}
holds for all $P=(p_{1}, \dots, p_{n})\in\mathcal{G}_{n},\,\,\, Q=(q_{1}, \dots, q_{m})\in\mathcal{G}_{m}$  where  \\
$P\ast Q=(p_{1}q_{1}, \dots, p_{1}q_{m}, \dots, p_{n}q_{1}, \dots, p_{n}q_{m})\in\mathcal{G}_{nm}$. Finally, we say that
an information measure  $\left(I_{n}\right)$ has the \emph{sum property}, if there exists a function
$f:I\to\mathbb{R}$ such that

\begin{equation}\label{Eq1.1.7}
I_{n}(p_{1}, \dots, p_{n})=\sum_{i=1}^{n}f(p_{i}) \qquad ((p_{1}, \dots, p_{n})\in\mathcal{G}_{n})
\end{equation}
for all $2\leq n\in\mathbb{N}$. Here (and through the paper) $I$ denotes the closed unit interval $[0,1]$ if $\mathcal{G}_{n}=\Gamma_{n}$ for all $2\leq n\in\mathbb{N}$
and the open unit interval $]0,1[$ if $\mathcal{G}_{n}=\Gamma^{\circ}_{n}$ for all $2\leq n\in\mathbb{N}$. Such a function $f$ satisfying \eqref{Eq1.1.7} is called a
\emph{generating function} of $(I_{n})$.

\subsection{The characterization problem and functional equations}

The properties listed above are of algebraic nature. This is the reason why they lead to functional equations. In this section,
we present how they imply the so-called parametric fundamental equation of information and the sum form functional equations. Following the
ideas of Dar\'oczy \cite{Dar69} (see also \cite{AczDar75}), suppose first that the information measure $(I_{n})$ is  \eqref{Eq1.1.5}  $\alpha$-recursive and
\eqref{Eq1.1.3} 3-semi-symmetric, and define the function $f$ on $I$ by
\[
 f(x)=I_{2}(x,1-x)
\]
and the set $D^{\circ}=\{(x,y)\, \vert \,  x,y,x+y \in I\}$, if $I=]0,1[$ and  $D=\{(x,y)\, \vert \,  x,y\in [0,1[, \, x+y \in I\}$ if $I=[0,1]$,
respectively. Let now $(x,y)\in D^{\circ}\cup D$ and  $n=3, p_{1}=1-x-y, p_{2}=y, p_{3}=x$ in \eqref{Eq1.1.5}.
Then we have that
\begin{multline*}
I_{3}(1-x-y,y,x)=I_{2}(1-x,x)+(1-x)^{\alpha}I_{2}\left(1-\frac{y}{1-x},\frac{y}{1-x}\right)\\
=f(x)+(1-x)^{\alpha}f\left(\frac{y}{1-x}\right)
\end{multline*}
which, by \eqref{Eq1.1.3}, implies that
\begin{equation}\label{Eq1.2.1}
f(x)+(1-x)^{\alpha}f\left(\frac{y}{1-x}\right)=
f(y)+(1-y)^{\alpha}f\left(\frac{x}{1-y}\right)
\end{equation}
holds on $D^{\circ}$ and on $D$, respectively. Functional equation \eqref{Eq1.2.1} is called the
\emph{parametric fundamental equation of information}, (in case $\alpha=1$ simply the \emph{fundamental equation of information}).

Furthermore, in case $\alpha=1$ and the domain $D$, its solutions $f:[0,1]\to\mathbb{R}$ satisfying
the additional requirements $f(0)=f(1), \, f\left(\frac{1}{2}\right)=1$
are the \emph{information functions}.

The role of the $\alpha$-recursivity is very important since, with the aid of this property, we can determine
the entire information measure from its initial element $I_{2}$. On the other hand, this idea shows the importance
of equation \eqref{Eq1.2.1}, as well.

The appearance of the sum form functional equations in the characterization problems of information measures is more evident.
Indeed, the \eqref{Eq1.1.6} $(\alpha, n,m)$- additivity and the \eqref{Eq1.1.7} sum property immediately imply the
functional equation
\begin{equation}\label{Eq1.2.2}
\sum_{i=1}^{n}\sum_{j=1}^{m}f(p_{i}q_{j})=\sum_{i=1}^{n}f(p_{i})+\sum_{j=1}^{m}f(q_{j})+(2^{1-\alpha}-1)\sum_{i=1}^{n}f(p_{i})\sum_{j=1}^{m}f(q_{j})
\end{equation}
for the generating function $f$.

As we shall see in the sections below, the solutions of \eqref{Eq1.2.1} and (in many cases) also of \eqref{Eq1.2.2} can be expressed by
the solutions of some well-known and well-discussed functional equations. In what follows we remind the reader some basic facts from
this part of the theory of functional equations.

\subsection{Prerequisites from the theory of functional equations}

All the results of this subsection can be found in the monographs Acz\'{e}l \cite{Acz66} and Kuczma \cite{Kuc09}.

Let $A\subset \mathbb{R}$ be an arbitrary nonempty set and
\[
\mathcal{A}=\left\{(x, y)\in\mathbb{R}^{2}\, \vert\,  x, y, x+y\in A\right\}.
\]
A function $a\colon I\to\mathbb{R}$ is called \emph{additive on $A$}, if for all $(x, y)\in\mathcal{A}$
\begin{equation}\label{Eq1.3.1}
a(x+y)=a(x)+a(y).
\end{equation}
If $A=\mathbb{R}$, then the function $a$ will be called simply \emph{additive}.
It is well-known that the solutions of the equation above, under some
mild regularity condition, are of the form
\[
 a(x)=cx \qquad \left(x\in I\right),
\]
with a certain real constant $c$. For example, it is true that those additive functions which are
bounded above or below on a set of positive Lebesgue measure have the form
\[
 a(x)=cx \qquad \left(x\in\mathbb{R}\right)
\]
with some $c\in\mathbb{R}$.
It is also known, however, that there are additive functions the graph of which is dense in the plain.
A great number of basic functional equations can easily be reduced to \eqref{Eq1.3.1}. In the following, we list
some of them.

Let
\[
\mathcal{M}=\left\{(x,y)\in\mathbb{R}^{2}\, \vert\,  x, y, xy\in A\right\}.
\]
A function $m\colon A\to\mathbb{R}$ is called
\emph{multiplicative on $A$}, if for all $(x, y)\in\mathcal{M}$
\begin{equation*}
m(xy)=m(x)m(y).
\end{equation*}
If $A=\mathbb{R_{+}}$ or $A=\mathbb{R_{++}}$ then the function $m$ is called simply \emph{multiplicative}.

Furthermore, we say that the function
$\ell\colon A\to\mathbb{R}$ is \emph{logarithmic on $A$} if for any
$(x,y)\in\mathcal{M}$,
\begin{equation*}
\ell(xy)=\ell(x)+\ell(y)
\end{equation*}

The functional equation
\begin{equation}\label{Eq1.3.2}
\varphi(xy)=x\varphi(y)+y\varphi(x)
\end{equation}
has an important role in the following and it can easily be reduced to the functional equation of logarithmic functions by introducing
the function $\ell(x)=\frac{\varphi(x)}{x}$. Finally, we will use functions $d:\mathbb{R}\to\mathbb{R}$ that are both additive and they are
solutions of functional equation \eqref{Eq1.3.2}, that is,
\begin{equation*}
d(xy)=xd(y)+yd(x)
\end{equation*}
is also satisfied for all $x,y\in\mathbb{R}$. This kind of functions are called \emph{real derivations}. Their complete description can be found in
Kuczma \cite{Kuc09} from which it turns out the somewhat surprising fact that \emph{there are non-identically zero real derivations}. Of course, if a
real derivation bounded from one side on a set of positive Lebesgue measure then it must be identically zero, otherwise its graph is dense in the plain.

In the subsequent sections it will occur that the equations introduced above
are fulfilled only on restricted domains. Most of these cases it can be proved that the functions in question
are the restrictions of some functions which satisfy the above equations on its natural domains.
The results of this type are the so-called \emph{extension theorems}, and the first classical ones are due to
Acz\'el--Erd\H os \cite{AczErd65} and Dar\'{o}czy--Losonczi \cite{DarLos77}. As a typical and important example,
we cite the following extension theorem (see \cite{DarLos77}).

\begin{thmszn}
Assume that the function
$a_{0}:[0,1]\to\mathbb{R}$ is additive on $[0,1]$. Then there
exists a uniquely determined function $a:\mathbb{R}\to\mathbb{R}$ which is
additive on $\mathbb{R}$ such that
\[
 a_{0}(x)=a(x)
\]
holds for all $x\in [0,1]$.
\end{thmszn}

Since all the other functional equations mentioned above in this subsection can be reduced to \eqref{Eq1.3.1} we can easily
get extension theorems for them as consequences, and their regular (say bounded on a set of positive Lebesgue measure)
solutions can also be obtained easily. In particular, the typical regular (say bounded from one side on a set of positive Lebesque measure)
solutions $\varphi:[0,+\infty[ \to\mathbb{R}$ of \eqref{Eq1.3.2} are of the form $\varphi(x)=cx\log_{2}(x)$ for all $0\leq x\in\mathbb{R}$
and for some $c\in\mathbb{R}$.

\section{Results on the fundamental equation of information and on the sum form equations}

\subsection{Information functions}

The first characterization theorem concerning the Shannon entropy (the case $\alpha=1$) considered on $\Gamma_{n}$
is due to Shannon himself, see \cite{Sha48}.
The second one, which is more abstract and mathematically well-based, can be found in
Khinchin \cite{Khi59}. In 1956, Faddeev succeed to reduce the
system of axioms used by the two previous authors, see \cite{Fad56}.
Faddeev assumed only symmetry, the normalization property, recursivity and
that the function $f\colon [0,1]\to\mathbb{R}$ defined by
\[
 f(x)=I_{2}(x,1-x) \qquad (x,y\in [0,1])
\]
is continuous. After that, the regularity assumption in the result of Faddeev was replaced by
weaker and weaker assumptions. For example, together with the above three algebraic properties
Tverberg \cite{Tve58} assumed (Lebesgue) integrability,
Lee \cite{Lee64} measurability, Dar\'oczy \cite{Dar69} continuity at zero ('small for small probabilities'),
and Diderrich \cite{Did86} boundedness on a set of positive measure, and they
showed that the above properties determine uniquely the Shannon entropy. We mention here the result of
Kendall \cite{Ken64} and Borges \cite{Bor67} who suppose monotonicity on the interval $[0,1/2[$ and proved
the same.

Concerning the characterization of the Shannon entropy, an 1969 paper of Dar\'{o}czy \cite{Dar69} meant a breakthrough.
He recognized that this characterization problem is equivalent with finding information functions that are identical with the
Shannon information function $S$ defined by
\[
S(x)=x\log_{2}(x)+(1-x)\log_{2}(1-x) \qquad (x\in I).
\]

An other important contribution was to find the general form of information functions (see  \cite{AczDar75})
which is the following.

\begin{thmszn}
A function $f:[0,1]\to\mathbb{R}$ is an information function if, and only if,
\begin{equation}\label{Eq2.1.1}
f(x)=\varphi(x)+\varphi(1-x) \quad \left(x\in [0, 1]\right)
\end{equation}
with some function $\varphi:[0,+\infty[\rightarrow \mathbb{R}$ satisfying the functional equation
\begin{equation}\label{Eq2.1.2}
\varphi(xy)=x\varphi(y)+y\varphi(x) \qquad (x,y\in [0,+\infty[)
\end{equation}
and $\varphi\left(\frac{1}{2}\right)=\frac{1}{2}$.
\end{thmszn}

The proof of this theorem is based on some results and ideas of purely algebraic nature in Jessen, Karpf,
and Thorup \cite{JesKarTho68} on the cocycle equation
\[
F(x+y,z)+F(x,y)=F(x,y+z)+F(y,z)
\]
that is satisfied, provided that
\[
F(x,y)=(x+y)f\left(\frac{y}{x+y}\right) \qquad (x,y\in \mathbb{R}_{+}, x+y\in \mathbb{R}_{++})
\]
where $f$ is an information function. Supposing that

\begin{equation*}
f(x)+(1-x)f\left(\frac{y}{1-x}\right)=
f(y)+(1-y)f\left(\frac{x}{1-y}\right)
\end{equation*}
holds only on the open domain $D^{\circ}=\{(x,y): x,y,x+y \in ]0,1[\}$
for the unknown function $f:]0,1[\to\mathbb{R}$, Maksa and Ng \cite{MakNg86} proved that $f(x)=\varphi(x)+\varphi(1-x)+ax$
for all $x\in ]0,1[$ and for some function $\varphi:[0,+\infty[\rightarrow \mathbb{R}$ satisfying functional equation
\eqref{Eq2.1.2} and for some $a\in\mathbb{R}$.

Obviously, if $\varphi(x)=-x\log_{2}x$, $x\in [0, +\infty[$ then
$\varphi\left(\frac{1}{2}\right)=\frac{1}{2}$, $\varphi$ satisfies \eqref{Eq2.1.2},
and \eqref{Eq2.1.1} implies that $f=S$. However, as it was pointed out in Acz\'el \cite{Acz81},
$f$ does not determine $\varphi$ unambiguously by \eqref{Eq2.1.1}.
Indeed, if $d:\mathbb{R}\rightarrow\mathbb{R}$  is a real derivation, that is,
$d$ satisfies both functional equations
\[
d(x+y)=d(x)+d(y) \,\,\,\,\, \text{and} \,\,\,\,\, d(xy)=xd(y)+yd(x)
\]
then \eqref{Eq2.1.1} and \eqref{Eq2.1.2} hold also with $\varphi +d$ instead of $\varphi$, moreover
$(\varphi +d)\left(\frac{1}{2}\right)=\frac{1}{2}$ is valid, as well. Thus, since there are non-identically zero
real derivations, the function $\varphi$ in \eqref{Eq2.1.1} does not inherit the regularity properties of $f$. So even for
very regular $f$ the function $\varphi$ may be very irregular. This is the main difficulty in deriving the regular solutions from the
general one.

The first successful attempt in this direction is due to Dar\'oczy \cite{Dar76}. 
By his observation, \eqref{Eq2.1.1} and \eqref{Eq2.1.2}
imply that
\begin{multline}\label{Eq2.1.3}
(x+y)f\left(\frac{y}{x+y}\right)=\varphi(x)+\varphi(y)-\varphi(x+y) 
\\
(x,y\in \mathbb{R}_{+}, x+y\in \mathbb{R}_{++}).
\end{multline}
If $f$ is (say) continuous then, for all fixed $y\in\mathbb{R}_{+}$ the difference functions $x\mapsto \varphi(x+y)-\varphi(x), \,\,x\in\mathbb{R}_{+}$
so are. Therefore, by a theorem of de Bruijn  \cite{Bru51}, $\varphi$ is a sum of a continuous and an additive function. It is not difficult to
show that the additive function is a real derivation and the other summand is a continuous solution of \eqref{Eq2.1.2}.

This is the point at which the stability idea first appeared in the investigation. Namely, supposing that the information function $f$
is bounded by a positive real number $\varepsilon$, \eqref{Eq2.1.3} implies that
\[
\left|\varphi(x)+\varphi(y)-\varphi(x+y)\right|\leq \varepsilon  \qquad (x,y\in \mathbb{R}_{+}, x+y\leq 1),
\]
that is, the Chauchy difference of $\varphi$ is bounded on a triangle. While de Bruijn type theorem is not true for
this case we could apply the stability theory in Maksa \cite{Mak80} to determine the bounded information functions
by giving a new and short proof of Diderrich's theorem published in \cite{Did86}.

At this point, we have to highlight the problem of nonnegative information functions. First of all,
we emphasis that the requirment of the nonnegativity for an information function is very natural from
information theoretical point of view, since $f(x)$ is the measure of information
belonging to the probability distribution $\left\{x, 1-x\right\},\,\, x\in [0,1]$. On the other hand, the
one-sided boundedness is important also from theoretical point of view, as well. Indeed, the solutions of
the Cauchy equation  \eqref{Eq1.3.1} bounded below or above on a set of positive Lebesgue measure are continuous linear functions.
Therefore, it was natural to expect that something similar is true for the information functions that are bounded from one side
(say nonnegative on $[0,1]$). Indeed, it was conjectured in Acz\'el--Dar\'oczy \cite{AczDar75} (supported by the partial result
Dar\'oczy--K\'atai \cite{DarKat70} by which the nonnegative information functions coincide with the Shannon one at the rational
points of $[0,1]$) that the only nonnegative information function is the Shannon one. The following counter example in Dar\'oczy--Maksa \cite{DarMak80}
however disproves this conjecture since, with any non-identically zero real derivation $d$, the function $f_{0}$ defined by

\[
f_{0}(x)=\left\{
\begin{array}{lcl}
S(x)+\dfrac{d(x)^{2}}{x(1-x)}, &\text{if}& x\in ]0, 1[ \\
0, &\text{if}& x\in\left\{0, 1\right\}
\end{array}
\right.
\]
is a nonnegative information function different from $S$. Of course there are positive results, as well. For example,
it is also proved in \cite{DarMak80} that $S(x)\leq f(x)$ for all nonnegative information function $f$ and for all $x\in [0,1]$.
An other one is about the set $K(f)=\left\{x\in [0, 1] \vert f(x)=S(x)\right\}$ which was introduced by Lawrence \cite{Law81}
and called the Shannon kernel of the nonnegative information function $f$. It is proved in Gselmann--Maksa \cite{GseMak09a}
that $K(f)$ has the form $[0,1]\cap L_{f}$ where $L_{f}$ is a subfield of $\mathbb{R}$
containing the square roots of its non-negative elements. Furthermore, if $K$ denotes the intersection of all
Shannon kernels (belonging to nonnegative information functions) then all the elements of $K$ are algebraic
over $\mathbb{Q}$ and $K$ contains all the algebraic elements of $[0,1]$ of degree at most $3$. Our first open problem is related to
these latter facts.

\begin{Opp}
Prove or disprove that all algebraic elements of the closed interval $[0,1]$ is contained by $K$, in other words
any nonnegative information function coincides with $S$ at the algebraic points of the closed unit interval.
\end{Opp}

We remark that Lawrence's conjecture in \cite{Law81} is affirmative.

The last sentences of this subsection are devoted to the case $\alpha\neq 1$ which is much simpler than the case $\alpha=1$. Indeed,
in \cite{Dar70}, Dar\'oczy determined all the solutions $f:[0,1]\to\mathbb{R}$ of \eqref{Eq1.2.1} satisfying the additional requirements
$f(0)=f(1), \, f\left(\frac{1}{2}\right)=1$. Thus he characterized the entropy of degree $\alpha$ on $\Gamma_{n}$ by using purely algebraic properties:
semi-symmetry, normalization and $\alpha$-recursivity. Since then, these results have been extended to the open domain case, as well (see e.g. the sections
about the stability).

\subsection{Sum form equations}

As we have seen earlier, the sum form equation \eqref{Eq1.2.2} is the consequence of the $(\alpha, n,m)$-additivity and the sum property. In connection with the
characterization properties discussed above, we should remark here the following implication: the sum property follows from the symmetry \eqref{Eq1.1.2}
and $\alpha$-recursivity  \eqref{Eq1.1.5}, as it is shown in \cite{AczDar75}.

In several characterization theorems for the entropy of degree $\alpha$ based on $(\alpha, n,m)$-additivity and the sum property, an additional regularity
condition was supposed for the generating function $f$ and also on the parameters $\alpha, \, n$ and $m$. We list some of results of these type in chronological order.

We begin with the Shannon case $\alpha=1$. Chaundy--McLeod \cite{ChaMcL60} proved that, if $f:[0,1]\to\mathbb{R}$ is continuous and
\begin{equation}\label{Eq2.2.1}
\sum_{i=1}^{n}\sum_{j=1}^{m}f(p_{i}q_{j})=\sum_{i=1}^{n}f(p_{i})+\sum_{j=1}^{m}f(q_{j})
\end{equation}
holds for all $(p_{1}, \ldots, p_{n})\in\Gamma_{n},\,\,(q_{1}, \ldots, q_{m})\in\Gamma_{m}$ and
for all $n\geq 2, m\geq 2$ then
\begin{equation}\label{Eq2.2.2}
f(x)=cx\log_{2}(x) \qquad (x\in [0,1])
\end{equation}
with some $c\in\mathbb{R}$. The same was proved by Acz\'el and Dar\'oczy \cite{AczDar63} supposing that $f$ is continuous and
\eqref{Eq2.2.1} holds for all $n=m\geq 2$. Dar\'oczy \cite{Dar71} determined the measurable solutions $f$ supposing that $n=3, m=2, f(1)=0$.
Dar\'oczy and J\'arai \cite{DarJar79} found the measurable solutions of \eqref{Eq2.2.1} in the case $n=m=2$ discovering solutions that are not
solutions when $n\geq 3$ or $m\geq 3$. This was one of the starting point of developing the regularity theory of functional
equations (see J\'arai \cite{Jar05}). In Maksa \cite{Mak81}, the solutions bounded from on a set of positive Lebesgue measure of \eqref{Eq2.2.1} were determined.
These are the same as in the continuous case (see \eqref{Eq2.2.2}) while it was also shown that the supposition of the  one-sided boundedness does not
lead to the same result. Counterexample can be given by real derivations (see Maksa \cite{Mak89}). Connected with these investigations the following problem
is still open.

\begin{Opp}
Find the general solution of equation \eqref{Eq2.2.1} for a fixed pair $(n,m), \, n\geq 2, m\geq 2$, particularly find all functions $f:I\to\mathbb{R}$
satisfying the functional equation
\begin{multline*}
f(xy)+f((1-x)y)+f(x(1-y))+f((1-x)(1-y))\\
=f(x)+f(1-x)+f(y)+f(1-y)
\end{multline*}
for all $x,y\in I$.
\end{Opp}

A partial result can be found in Losonczi--Maksa \cite{LosMak82}.

As we have already mentioned, in the characterization theorems for the entropy of degree $\alpha$ based on $(\alpha, n,m)$-additivity and
the sum property, an additional regularity condition was supposed for the generating function $f$. Now we present here an exceptional case
(see Maksa \cite{Mak89}) in which all the conditions refer to the information measure itself and there is no condition on the
generating function. The stability idea appears again. Indeed, suppose that the information measure $(I_{n})$ is $(1, n,m)$-additive for some
$n\geq 3, \, m\geq 2$, has the sum property with generating function $f:[0,1]\to\mathbb{R}$ and $I_{3}$ is bounded by the real number $K$, that is,
\begin{equation}\label{Eq2.2.3}
|I_{3}(p_{1},p_{2},p_{3})|\leq K \qquad ((p_{1},p_{2},p_{3})\in \Gamma_{3}).
\end{equation}
Let $x,y\in [0,1]$ such that $x+y\leq 1$ and apply \eqref{Eq2.2.3} to the probability distributions $(x,y,1-x-y)\in \Gamma_{3}$ and then
to $(x+y,1-x-y,0)\in \Gamma_{3}$, respectively to get that
\[
|I_{3}(x,y,1-x-y)|\leq K \,\, \text{and} \,\, |I_{3}(x+y,1-x-y,0)|\leq K.
\]
Therefore, because of the triangle inequality, for the generating function $f$, we have that
\[
|f(x+y)-f(x)-f(y)+f(0)|\leq 2K,
\]
that is, the stability inequality holds for the function $f-f(0)$ on a triangle. The details together with the consequences are in \cite{Mak89}.

The brief history of the case $\alpha\neq 1$ follows. The continuous solutions, supposing that \eqref{Eq1.2.2} holds for all $n\geq 2, m\geq 2$ were
determined by Behara and Nath \cite{BehNat73}, Kannappan \cite{Kan74} and Mittal \cite{Mit70} independently of each other. They found that the
continuous solutions either a sum of a continuous additive function and a constant or the sum of a continuous additive function and a continuous multiplicative
function (power function). The same was proved by Losonczi \cite{Los81} supposing that \eqref{Eq1.2.2} holds for a fixed pair $(n,m), \,\,n\geq 3, m\geq 2$
and the generating function $f$ in \eqref{Eq1.2.2} is measurable. Contrary to the case $\alpha=1$, in the case $\alpha\neq 1$ the general solution has been determined
(see Losonczi-Maksa \cite{LosMak82} and Maksa \cite{Mak87}) supposing that $n\geq 3$ and $m\geq 2$ are fixed. Characterization theorems for the entropy of degree
$\alpha$ can easily be derived from these results (see \cite{Mak89}).

In the case $\alpha\neq 1$, with the definition $g(p)=p+(2^{1-\alpha}-1)f(p), \,\, p\in I$, equation \eqref{Eq1.2.2} can be reduced to equation

\begin{equation}\label{Eq2.2.4}
\sum_{i=1}^{n}\sum_{j=1}^{m}g(p_{i}q_{j})=\sum_{i=1}^{n}g(p_{i})\sum_{j=1}^{m}g(q_{j}).
\end{equation}

The general solution of which is not known when $n=m=2$. Therefore we formulate the following open problem.

\begin{Opp}
Find all functions $g:I\to\mathbb{R}$ satisfying the functional equation
\begin{multline*}
g(xy)+g((1-x)y)+g(x(1-y))+g((1-x)(1-y))\\
=\left(g(x)+g(1-x)\right)\left(g(y)+g(1-y)\right)
\end{multline*}
for all $x,y\in I$.
\end{Opp}
A partial result is proved in Losonczi \cite{Los89}.

Further investigations related to sum form equations on open domain or for functions in several variables can be found among others in
Losonczi \cite{Los85} and in the survey paper Ebanks--Kannappan--Sahoo--Sander \cite{EbaKanSahSan97}.

\section{Stability problems}

During one of his talk, held at the University of
Wisconsin S.~Ulam posed several problems.
One of these problems has became the cornerstone
of the stability theory of functional equations,
see \cite{Ula64}.
Ulam's problem reads as follows.

Let $(G, \circ)$ be a group and
$(H, \ast)$ be a metric group with the metric
$d$.
Let $\varepsilon\geq 0$ and $f:G\rightarrow\mathbb{R}$ be
a function such that
\[
d\left(f(x\circ y), f(x)\ast f(y)\right)\leq \varepsilon
\]
holds for all $x, y\in G$.
Is it true that there exist $\delta\geq 0$ and a function
$g:G\rightarrow\mathbb{R}$ such that
\[
g(x\circ y)=g(x)\ast g(y),
\quad
\left(x, y\in G\right)
\]
so that
\[
d\left(f(x), g(x)\right)\leq \delta
\]
holds for all $x\in G$?

This question was first answered in 1941 by
D.~H.~Hyers by proving the following theorem, see \cite{Hye41}.

\begin{thmszn}
Let $\varepsilon\geq 0$, $X, Y$ be Banach spaces and
$f:X\rightarrow Y$ be a function.
Suppose that
\[
\left\|f(x+y)-f(x)-f(y)\right\|\leq \varepsilon
\]
holds for all $x, y\in X$.
Then,  for all $x\in X$, the limit
\[
a(x)=\lim_{n\rightarrow\mathbb{R}}\frac{f(2^{n}x)}{2^{n}}
\]
does exist, the function $a:X\rightarrow\mathbb{R}$ is additive
on $X$, i.e.,
\[
a(x+y)=a(x)+a(y)
\]
holds for all $x, y\in X$, furthermore,
\[
\left\|f(x)-a(x)\right\|\leq \varepsilon
\]
is fulfilled for arbitrary $x\in X$.
Additionally, the function $a:\mathbb{R}\rightarrow\mathbb{R}$ is
uniquely determined by the above formula.
\end{thmszn}

The above theorem briefly expresses the following.
Assume that $X, Y$ are Banach spaces and the function
$f:X\rightarrow Y$ satisfies the additive Cauchy equation
only 'approximatively'.
Then there exists a unique additive function
$a:X\rightarrow Y$ which is 'close' to the function $f$.
Since 1941 this result has been extended and generalized in a several
ways. Furthermore,
Ulam's problem can obviously be raised concerning not only
the Cauchy equation but also in connection with other equations, as well.
For further result consult the monograph Hyers--Isac--Rassias \cite{HyeIsaRas98}.

For instance, the stability problem of the
multiplicative Cauchy equation highlighted a
new phenomenon, which is nowadays called
\emph{superstability}.
In this case the so--called stability inequality
implies that the function in question is either
bounded or it is the exact solution of the
functional equation in question, see Baker \cite{Bak80}.

In this work we will meet an other notion, namely the
\emph{hyperstability}.
In this case, from the stability inequality,
we get that the function in question can be nothing else
than the exact solution of the functional equation in question,
see, e.g. Maksa--Páles \cite{MakPal01}.

Since the above result of D.~H.~Hyers appeared, the stability theory of functional
equations became a rapidly developing area.
Presently, in the theory of stability there exist several methods, e.g., the
Hyers' method (c.f. Forti \cite{For95}), the method of invariant means
(see Székelyhidi \cite{Sze82, Sze95}) and the method that is based on separation theorems
(see Badora--Ger--Páles \cite{BadGerPal03}).

As we will see in the following subsections, in case of the functional equations, we will deal with,
\emph{none of the above methods will work}. More precisely, in some cases the method of invariant means
is used. However, basically we have to develop new ideas to prove stability type theorems for the
functional equations, we mentioned in the introduction.
Concerning topic of invariant means, we offer the expository paper Day \cite{Day57}.
Although the only result needed from \cite{Day57} is,
that on every commutative semigroup there
exist an invariant mean, that is,
every commutative semigroup is \emph{amenable}.

The aim of this paper is to investigate the stability
of some functional equations that appear in the theory of information.
Firstly, we will investigate the above problem concerning
the parametric fundamental equation of information.
The main results and also the applications will
be listed in the subsequent subsections.
We will prove stability, superstability and hyperstability
according to the value of the parameter $\alpha$.
The results, we will present can be found in
Gselmann \cite{Gse09a, Gse09b, Gse09, Gse10} and in
Gselmann--Maksa \cite{GseMak09}.

Concerning the stability of the parametric fundamental equation of information, the first result was
the stability of equation \eqref{Eq1.1.6} on the set $D$,
assuming that $1\neq \alpha>0$ (see Maksa \cite{Mak08}). Furthermore
the stability constant, got in that paper is much smaller than that of our.
However, the method, used in Maksa \cite{Mak08} does not work if
$\alpha=1$ or $\alpha\leq 0$ or if we consider the problem on the open domain.

After that, it was proved that equation \eqref{Eq1.1.6} is stable in the sense
of Hyers and Ulam on the set $D^{\circ}$ as well as on $D$, assuming that $\alpha\leq 0$
(see \cite{GseMak09}). After that it turned out that this method is appropriate
to prove superstability  in case $1\neq \alpha>0$.
This enabled us to give a unified proof for
the stability problem of equation \eqref{Eq1.1.6}.
Finally, using a different approach,
in \cite{Gse09a} it was showed that in case $\alpha<0$, the parametric
fundamental equation of information is hyperstable on $D^{\circ}$ as well as on $D$.

\subsection{The cases $\alpha=0$ and $0<\alpha\neq 1$}

In this part of the paper we will investigate the stability of the parametric
fundamental equation of information in case for the parameter $\alpha$,
$\alpha=0$ or $0<\alpha\neq 1$ holds. The method, we will use during the proofs
were firstly developed for the case $\alpha<0$. However, it turned out that this approach
works in this case also. The results we will present here can
be found in \cite{Gse09, Gse09b} and also in \cite{Mak08}.

\begin{thm}\label{T3.1.1}
Let $\alpha, \varepsilon\in\mathbb{R}$ be fixed,
$1\neq\alpha\geq 0, \varepsilon\geq 0$.
Suppose that the function $f:]0, 1[\to\mathbb{R}$
satisfies the inequality
\begin{equation}\label{Eq3.1.1}
\left|f(x)+(1-x)^{\alpha}f\left(\frac{y}{1-x}\right)
-f(y)-(1-y)^{\alpha}f\left(\frac{x}{1-y}\right)
\right|\leq\varepsilon
\end{equation}
for all $(x, y)\in D^{\circ}$.
Then, in case $\alpha=0$,
there exists a logarithmic function
$l:]0, 1[\to\mathbb{R}$ and $c\in\mathbb{R}$ such that
\begin{equation}\label{Eq3.1.2}
\left|f(x)-\left[l(1-x)+c\right]\right|\leq K(\alpha)\varepsilon,
\quad \left(x\in ]0, 1[\right)
\end{equation}
furthermore, if $\alpha\notin \left\{0, 1\right\}$, there
exist $a, b\in\mathbb{R}$ such that
\begin{equation}\label{Eq3.1.3}
\left|f(x)-\left[ax^{\alpha}+b(1-x)^{\alpha}-b\right]\right|\leq K(\alpha)\varepsilon
\end{equation}
holds for all $x\in ]0, 1[$, where
\[
K(\alpha)=
\left|2^{1-\alpha}-1\right|^{-1}
\left(3+12\cdot2^{\alpha}+\frac{32\cdot3^{\alpha+1}}{\left|2^{-\alpha}-1\right|}\right).
\]
\end{thm}

\begin{proof}
Define the function $F$ on $\mathbb{R}^{2}_{++}$ by
\begin{equation}\label{Eq3.1.4}
F(u, v)=(u+v)^{\alpha}f\left(\frac{v}{u+v}\right).
\end{equation}
Then
\begin{equation}\label{Eq3.1.5}
F(tu, tv)=t^{\alpha}F(u, v) \quad \left(t, u, v \in \mathbb{R}_{++}\right)
\end{equation}
and
\begin{equation}\label{Eq3.1.6}
f(x)=F(1-x, x), \quad \left(x\in ]0,1[\right)
\end{equation}
furthermore, with the substitutions
\[
x=\frac{w}{u+v+w}, \quad y=\frac{v}{u+v+w} \quad \left(u, v, w\in \mathbb{R}_{++}\right)
\]
inequality \eqref{Eq3.1.1} implies that
\begin{multline}\label{Eq3.1.7}
\left|f\left(\frac{w}{u+v+w}\right)+\frac{(u+v)^{\alpha}}{(u+v+w)^{\alpha}}f\left(\frac{v}{u+v}\right)\right.
\\\left. -f\left(\frac{v}{u+v+w}\right)-\frac{(u+w)^{\alpha}}{(u+v+w)^{\alpha}}f\left(\frac{w}{u+w}\right)\right|
\leq \varepsilon
\end{multline}
whence, by \eqref{Eq3.1.4}
\begin{equation}\label{Eq3.1.8}
\left|F(u+v, w)+F(u, v)-F(u+w, v)-F(u, w)\right|\leq \varepsilon (u+v+w)^{\alpha}
\end{equation}
follows for all $u, v, w\in\mathbb{R}_{++}$.

In the next step we define the functions $g$ and $G$ on
$\mathbb{R}_{++}$ and on $\mathbb{R}_{++}^{2}$, respectively by
\begin{equation}\label{Eq3.1.9}
g(u)=F(u, 1)-F(1, u)
\end{equation}
and
\begin{equation}\label{Eq3.1.10}
G(u, v)=F(u, v)+g(v).
\end{equation}
We will show that
\begin{equation}\label{Eq3.1.11}
\left|G(u, v)-G(v, u)\right|\leq 3\varepsilon (u+v+1)^{\alpha}.
\quad \left(u, v\in \mathbb{R}_{++}\right)
\end{equation}
Indeed, with the substitution $w=1$, inequality \eqref{Eq3.1.8}
implies that
\begin{equation}\label{Eq3.1.12}
\left|F(u+v, 1)+F(u, v)-F(u+1, v)-F(u, 1)\right|\leq \varepsilon (u+v+1)^{\alpha}.
\end{equation}
Interchanging $u$ and $v$, it follows from \eqref{Eq3.1.12} that
\begin{multline*}
\left|-F(u+v, 1)-F(v, u)+F(v+1, u)-F(v, 1)\right|\leq \varepsilon (u+v+1)^{\alpha}
\\ \left(u, v\in\mathbb{R}_{++}\right). 
\end{multline*}
This inequality, together with \eqref{Eq3.1.12} and the triangle inequality
imply that
\begin{multline}\label{Eq3.1.13}
\left|F(u, v)-F(v, u)-F(u+1, v)-F(u, 1)+F(v+1, u)+F(v, 1)\right|
\\
\leq 2\varepsilon (u+v+1)^{\alpha}
\end{multline}
holds for all $u, v\in\mathbb{R}_{++}$.
On the other hand, with $u=1$, we get from \eqref{Eq3.1.8} that
\[
\left|F(1+v, w)+F(1, v)-F(1+w, w)-F(1, w)\right|\leq \varepsilon(1+v+w)^{\alpha}.
\]
Replacing here $v$ by $u$ and $w$ by $v$, respectively, we have that
\begin{multline*}
\left|F(u+1, v)+F(1, u)-F(v+1, u)-F(1, v)\right|\leq \varepsilon (u+v+1)^{\alpha}
\\ \left(u, v\in\mathbb{R}_{++}\right). 
\end{multline*}
Again, by the triangle inequality and the definitions
\eqref{Eq3.1.9} and \eqref{Eq3.1.10}, \eqref{Eq3.1.13} and the last inequality imply \eqref{Eq3.1.11}.

In what follows we will investigate the function $g$.
At this point of the proof we have to distinguish two cases.

\emph{Case I. ($\alpha=0$)}.In this case we will show that there exists a logarithmic function
$l:\mathbb{R}_{++}\to\mathbb{R}$ such that
\[
\left|g(u)-l(u)\right|\leq 6\varepsilon
\]
for all $u\in\mathbb{R}_{++}$.
Indeed, \eqref{Eq3.1.11} yields in this case that
\[
\left|G(u, v)-G(v, u)\right|\leq 3\varepsilon.
\quad \left(u, v\in \mathbb{R}_{++}\right)
\]
Due to \eqref{Eq3.1.5} and \eqref{Eq3.1.10} we obtain that
\[
G(tu, tv)=F(tu, tv)+g(tv)
=F(u, v)+g(tv)
=G(u, v)-g(v)+g(tv)
\]
that is,
\[
G(tu, tv)-G(u, v)=g(tv)-g(v),  \quad \left(t, u, v\in\mathbb{R}_{++}\right)
\]
therefore
\begin{multline}\label{Eq3.1.14}
\left|g(tv)-g(v)+g(u)-g(tu)\right|
=\left|G(tu, tv)-G(u, v)-G(tv, tu)+G(v, u)\right| \\
\leq \left|G(tu, tv)-G(tv, tu)\right|+\left|G(v, u)-G(u, v)\right|
\leq 6\varepsilon
\end{multline}
for all $t, u, v\in\mathbb{R}_{++}$. Now \eqref{Eq3.1.14} with the substitution
$u=1$ implies that
\[
\left|g(tv)-g(v)-g(t)\right|\leq 6\varepsilon
\]
holds for all $t, v\in\mathbb{R}_{++}$, since obviously $g(1)=0$.
This means that the function $g$ is approximately logarithmic on
$\mathbb{R}_{++}$. Thus there exists
a logarithmic function $l:\mathbb{R}_{++}\to\mathbb{R}$ such that
\[
\left|g(u)-l(u)\right|\leq 6\varepsilon
\]
holds for all $u\in \mathbb{R}_{++}$.

Furthermore,
\begin{multline}\label{Eq3.1.15}
\left|f(x)-l(1-x)-\left(f(1-x)-l(x)\right)\right| \\
=\left|F(1-x, x)-l(1-x)-F(x, 1-x)+l(x)\right| \\
=\left|F(1-x, x)+g(x)-g(x)-l(1-x)\right. 
\\
-F(x, 1-x)+g(1-x)-g(1-x)+l(x)\left.\right| \\
\leq \left|F(1-x, x)+g(x)-\left(F(x, 1-x)+g(1-x)\right)\right| \\
+\left|g(1-x)-l(1-x)\right|+\left|l(x)-g(x)\right| \\
=\left|G(1-x, x)-G(x, 1-x)\right|
+\left|g(1-x)-l(1-x)\right|+\left|l(x)-g(x)\right| \\
\leq 3\varepsilon+6\varepsilon+6\varepsilon=15 \varepsilon
\end{multline}
Define the functions $f_{0}$ and
$F_{0}$ on $]0, 1[$ and on $]0, 1[^{2}$, respectively, by
\[
f_{0}(x)=f(x)-l(1-x)
\]
and
\[
F_{0}(p, q)=f_{0}(p)+f_{0}(q)-f_{0}(pq)-f_{0}\left(\frac{1-p}{1-pq}\right)
\]
Due to \eqref{Eq3.1.15}
\begin{equation}\label{Eq3.1.16}
\left|f_{0}(x)-f_{0}(1-x)\right|\leq 15\varepsilon
\end{equation}
holds for all $x\in ]0, 1[$.
Furthermore, with the substitutions $x=1-p$, $y=pq$ ($p, q\in ]0, 1[$)
inequality \eqref{Eq3.1.1} implies, that
\begin{equation}\label{Eq3.1.17}
\left|f_{0}(1-p)+f_{0}(q)-
f_{0}(pq)-f_{0}\left(\frac{1-p}{1-pq}\right)\right|\leq\varepsilon
\end{equation}
is fulfilled for all $p, q\in ]0, 1[$.
Inequalities \eqref{Eq3.1.16}  and \eqref{Eq3.1.17} and the triangle inequality
imply that
\begin{equation}\label{Eq3.1.18}
\left|F_{0}(p, q)\right|\leq 16\varepsilon
\end{equation}
for all $p, q\in ]0, 1[$.
An easy calculation shows that
\begin{multline*}
f_{0}(p)-f_{0}(q)\\
=F_{0}(q, p)-F_{0}(p, q)+F_{0}\left(\frac{1-p}{1-pq}, p\right)-
f_{0}\left(1-\frac{1-p}{1-pq}\right)+f_{0}\left(\frac{1-p}{1-pq}\right)
\end{multline*}
therefore,
\begin{multline}\label{Eq3.1.19}
\left|f_{0}(p)-f_{0}(q)\right|\\
\leq
\left|F_{0}(q, p)\right|+\left|F_{0}(p, q)\right|+
\left|F_{0}\left(\frac{1-p}{1-pq}, p\right)\right|
\\
+
\left|f_{0}\left(1-\frac{1-p}{1-pq}\right)-f_{0}\left(\frac{1-p}{1-pq}\right)\right| \\
\leq 3\cdot 16\varepsilon+15\varepsilon=63\varepsilon
\end{multline}
holds for all $p, q\in ]0, 1[$.
With the substitution $q=\frac{1}{2}$ inequality \eqref{Eq3.1.19} implies that
\[
\left|f_{0}(p)-f_{0}\left(\frac{1}{2}\right)\right| \leq 63\varepsilon.
\quad \left(p\in ]0, 1[\right)
\]
Using the definition of the function $f_{0}$, we obtain that inequality
\[
\left|f(x)-l(1-x)-c\right|\leq 63\varepsilon
\]
is satisfied for all $x\in]0, 1[$, where $c=f_{0}\left(\frac{1}{2}\right)$.
Hence inequality \eqref{Eq3.1.2} holds, indeed.

\emph{Case II. ($1\neq \alpha\geq 0$).}Finally, we will prove that
there exists $c\in\mathbb{R}$ such that
\[
\left|g(x)-c(x^{\alpha}-1)\right|
\leq \frac{4\cdot 3^{\alpha+1}\varepsilon}{\left|2^{-\alpha}-1\right|}
\]
holds for all $x\in ]0, 1[$.

Due to inequalities \eqref{Eq3.1.4} and \eqref{Eq3.1.9},
\begin{multline*}
G(tu, tv)=F(tu, tv)+g(tv)
 =t^{\alpha}F(u, v)+g(tv)\\
 =t^{\alpha}G(u, v)-t^{\alpha}g(v)+g(tv),
\end{multline*}
that is,
\[
G(tu, tv)-t^{\alpha}G(u, v)=g(tv)-t^{\alpha}g(v)
\]
holds for all $t, v\in\mathbb{R}_{++}$.
Therefore,
\begin{multline}\label{Eq3.1.20}
\left|g(tv)-t^{\alpha}g(v)+t^{\alpha}g(u)-g(tu)\right| \\
=\left|G(tu, tv)-G(u, v)-G(tv, tu)+G(v, u)\right| \\
\leq \left|G(tu, tv)-G(tv, tu)\right|+\left|G(u, v)-G(v, u)\right| \\
\leq 3\varepsilon(t(u+v)+1)^{\alpha}+3\varepsilon(u+v+1)^{\alpha}
\end{multline}
holds for all $t, u, v\in\mathbb{R}_{++}$, where we used \eqref{Eq3.1.1}.
With the substitution $u=1$, \eqref{Eq3.1.20} implies that
\begin{multline}\label{Eq3.1.21}
\left|g(tv)-t^{\alpha}g(v)-g(t)\right|
\leq 3\varepsilon(t(v+1)+1)^{\alpha}+3\varepsilon(v+2)^{\alpha}
\quad \left(t, v\in\mathbb{R}_{++}\right)
\end{multline}
Interchanging $t$ and $v$ in \eqref{Eq3.1.21}, we obtain that
\begin{multline}\label{Eq3.1.22}
\left|g(tv)-v^{\alpha}g(t)-g(v)\right|
\leq 3\varepsilon(v(t+1)+1)^{\alpha}+3\varepsilon(t+2)^{\alpha}
\quad \left(t, v\in\mathbb{R}_{++}\right)
\end{multline}
Inequalities \eqref{Eq3.1.21}, \eqref{Eq3.1.22} and the triangle inequality
imply that
\begin{equation}\label{Eq3.1.23}
\left|t^{\alpha}g(v)+g(t)-v^{\alpha}g(t)-g(v)\right|
\leq B(t, v)
\end{equation}
is fulfilled for all $t, v\in\mathbb{R}_{++}$, where

\[
B(t, v)= 3\varepsilon(t(v+1)+1)^{\alpha}+3\varepsilon(v+2)^{\alpha}
 +3\varepsilon(v(t+1)+1)^{\alpha}+3\varepsilon(t+2)^{\alpha}.
\]

With the substitution $t=\frac{1}{2}$ and with the definition
$c=\frac{g\left(\frac{1}{2}\right)}{2^{-\alpha}-1}$, we obtain
\begin{equation}\label{Eq3.1.24}
\left|g(v)-c(v^{\alpha}-1)\right|\leq \frac{B\left(\frac{1}{2}, v\right)}{\left|2^{-\alpha}-1\right|}
\end{equation}
for all $v\in\mathbb{R}_{++}$.

Let us observe that
\[
\left|B(t, v)\right|\leq 4\cdot 3^{\alpha+1}\varepsilon
\]
holds, if $t, v\in ]0, 1[$.
Thus
\begin{equation}\label{Eq3.1.25}
\left|g(v)-c(v^{\alpha}-1)\right|\leq \frac{B\left(\frac{1}{2}, v\right)}{\left|2^{-\alpha}-1\right|}
\leq \frac{4\cdot 3^{\alpha+1}\varepsilon}{\left|2^{-\alpha}-1\right|}
\end{equation}
for all $v\in ]0, 1[$.
Therefore \eqref{Eq3.1.6}, \eqref{Eq3.1.10}, \eqref{Eq3.1.11}, \eqref{Eq3.1.25} and the
triangle inequality imply that
\begin{multline}\label{Eq3.1.26}
\left|f(x)-c(1-x)^{\alpha}+c-\left(f(1-x)-cx^{\alpha}+c\right)\right| \\
=\left|F(1-x, x)-c(1-x)^{\alpha}+c-\left(F(x, 1-x)-cx^{\alpha}+c\right)\right| \\
\leq \left|F(1-x, x)+g(x)-F(x, 1-x)-g(1-x)\right| \\
+ \left|g(x)-c(x^{\alpha}-1)\right|+\left|g(1-x)-c((1-x)^{\alpha}-1)\right| \\
=\left|G(1-x, x)-G(x, 1-x)\right| \\
+ \left|g(x)-c(x^{\alpha}-1)\right|+\left|g(1-x)-c((1-x)^{\alpha}-1)\right| \\
\leq 3\cdot 2^{\alpha}\varepsilon+\frac{8\cdot 3^{\alpha+1}\varepsilon}{\left|2^{-\alpha}-1\right|}
\end{multline}
holds for all $x\in ]0, 1[$.

As in the previous cases, we define the functions $f_{0}$ and
$F_{0}$ on $]0, 1[$ and on $]0, 1[^{2}$ by
\begin{equation}\label{Eq3.1.27}
f_{0}(x)=f(x)-c(1-x)^{\alpha}
\end{equation}
and
\begin{equation}\label{Eq3.1.28}
F_{0}(p, q)=f_{0}(p)+p^{\alpha}f_{0}(q)-f_{0}(pq)-
(1-pq)^{\alpha}f_{0}\left(\frac{1-p}{1-pq}\right),
\end{equation}
respectively.
Then \eqref{Eq3.1.1}, \eqref{Eq3.1.26} and \eqref{Eq3.1.27} imply that
\begin{equation}\label{Eq3.1.29}
\left|f_{0}(x)+(1-x)^{\alpha}f_{0}\left(\frac{y}{1-x}\right)
-f_{0}(y)-(1-y)^{\alpha}f_{0}\left(\frac{x}{1-y}\right)\right|\leq \varepsilon
\end{equation}
for all $(x, y)\in D^{\circ}$ and
\begin{equation}\label{Eq3.1.30}
\left|f_{0}(x)-f_{0}(1-x)\right|\leq 3\cdot 2^{\alpha}\varepsilon+\frac{8\cdot 3^{\alpha+1}\varepsilon}{\left|2^{-\alpha}-1\right|}.
\quad \left(x\in ]0, 1[\right)
\end{equation}
Furthermore, with the substitutions $x=1-p$, $y=pq$ ($p, q\in ]0, 1[$), \eqref{Eq3.1.29}
implies that
\begin{equation}\label{Eq3.1.31}
\left|f_{0}(1-p)+p^{\alpha}f_{0}(q)-
f_{0}(pq)-(1-pq)^{\alpha}f_{0}\left(\frac{1-p}{1-pq}\right)\right|\leq\varepsilon
\end{equation}
holds for all $p, q\in ]0, 1[$.
Thus \eqref{Eq3.1.30} and \eqref{Eq3.1.31} and the triangle
inequality imply that
\[
\left|F_{0}(p, q)\right|\leq \varepsilon+3\cdot 2^{\alpha}\varepsilon+\frac{8\cdot 3^{\alpha+1}\varepsilon}{\left|2^{-\alpha}-1\right|} .
\quad \left(x\in ]0, 1[\right)
\]
Similarly to the previous case, it is easy to see that the identity
\begin{multline}\label{Eq3.1.32}
f_{0}(p)\left[q^{\alpha}+(1-q)^{\alpha}-1\right]-f_{0}(q)\left[p^{\alpha}+(1-p)^{\alpha}-1\right] \\
=F_{0}(q, p)-F_{0}(p, q)\\-(1-pq)^{\alpha}
\left[F_{0}\left(\frac{1-q}{1-pq}, p\right)+
f_{0}\left(1-\frac{1-p}{1-pq}\right)
-f_{0}\left(\frac{1-p}{1-pq}\right)\right]
\end{multline}
is satisfied for all $p, q\in ]0, 1[$. Therefore
\begin{multline*}
\left|f_{0}(p)-\frac{f_{0}(q)}{q^{\alpha}+(1-q)^{\alpha}-1}
\left[p^{\alpha}+(1-p)^{\alpha}-1\right]\right| \\
\leq \left|q^{\alpha}+(1-q)^{\alpha}-1\right|^{-1} \times  
\\
\times \left(3\left(\varepsilon+3\cdot 2^{\alpha}\varepsilon+\frac{8\cdot 3^{\alpha+1}\varepsilon}{\left|2^{-\alpha}-1\right|}\right)+
3\cdot 2^{\alpha}\varepsilon+\frac{8\cdot 3^{\alpha+1}\varepsilon}{\left|2^{-\alpha}-1\right|}\right)
\end{multline*}

for all $p, q\in ]0, 1[$.
In view of \eqref{Eq3.1.27}, with $q=\frac{1}{2}$ with the definitions
\[
a=f_{0}\left(\frac{1}{2}\right)\left(2^{1-\alpha}-1\right)^{-1}\quad  \text{and } \quad
b=a+c,
\]
this inequality implies that
\begin{equation}\label{Eq3.1.33}
\left|f(p)-\left[ap^{\alpha}+b(1-p)^{\alpha}-b\right]\right|\leq K(\alpha)\varepsilon
\end{equation}
holds for all $p\in ]0, 1[$, where
\[
K(\alpha)=
\left|2^{1-\alpha}-1\right|^{-1}
\left(3+12\cdot2^{\alpha}+\frac{32\cdot3^{\alpha+1}}{\left|2^{-\alpha}-1\right|}\right),
\]
which had to be proved.
\end{proof}

In the following theorem we shall prove that the parametric fundamental equation of information
is stable not only on $D^{\circ}$ but also
on $D$. During the proof of this theorem the following function
will be needed.
For all $1\neq \alpha >0$ we define the function $T(\alpha)$ by
\[
T(\alpha)= 3\cdot 2^{\alpha} +\frac{8\cdot 3^{\alpha+1}}{\left|2^{-\alpha}-1\right|}.
\]
Furthermore, the following relationship is fulfilled between $K(\alpha)$ and $T(\alpha)$
\[
K(\alpha)=\frac{4T(\alpha)+3}{\left|2^{1-\alpha}-1\right|}
\]
for all $1\neq \alpha>0$.

\begin{thm}\label{T3.1.2}
Let $\alpha, \varepsilon\in\mathbb{R}$ be fixed, $0\leq\alpha\neq 1$, $\varepsilon\geq 0$.
Suppose that the function $f:[0,1]\to\mathbb{R}$ satisfies inequality
\eqref{Eq3.1.1} for all $(x, y)\in D$.
Then, in case $\alpha\neq 0$ there
exist $a, b\in\mathbb{R}$ such that the function $h_{1}$
defined on $[0, 1]$ by
\[
h_{1}(x)=\left\{
\begin{array}{lcl}
0, & \text{if} & x=0\\
ax^{\alpha}+b(1-x)^{\alpha}-b, &\text{if}& x\in \left.]0, 1[\right. \\
a-b, & \text{if} & x=1
\end{array}
\right.
\]
is a solution of \eqref{Eq1.1.6} on $D$ and
\begin{equation}\label{Eq3.1.34}
\left|f(x)-h_{1}(x)\right|\leq \max\left\{K(\alpha), T(\alpha)+1\right\}\varepsilon
\quad \left(x\in [0, 1]\right)
\end{equation}
holds.
In case $\alpha=0$, there exists $c\in\mathbb{R}$ such
that the function $h_{2}$ defined on $[0, 1]$ by
\[
h_{2}(x)=\left\{
\begin{array}{lcl}
f(0), & \text{if} & x=0\\
c, &\text{if}& x\in \left.]0, 1[\right. \\
f(1), & \text{if} & x=1
\end{array}
\right.
\]
is a solution of \eqref{Eq1.1.6} on $D$ and
\begin{equation}\label{Eq3.1.35}
\left|f(x)-h_{2}(x)\right|\leq K(\alpha)\varepsilon.
\quad \left(x\in [0, 1]\right)
\end{equation}
is fulfilled.
\end{thm}
\begin{proof}
 An easy calculation shows that the functions
$h_{1}$ and $h_{2}$ are the solutions of equation \eqref{Eq1.1.6}
on $D$ in case $\alpha\neq 0$ and $\alpha=0$, respectively.

At first, we deal with the case $\alpha>0$.
Substituting $x=0$ into \eqref{Eq3.1.1} and with $y\to 0$ we obtain that
\[
\left|f(0)\right|\leq \varepsilon \leq K(\alpha) \varepsilon,
\]
that is, \eqref{Eq3.1.34} holds for $x=0$.
If $x\in ]0, 1[$, then inequality \eqref{Eq3.1.34} follows immediately from
Theorem \ref{T3.1.1}.
Furthermore, with the substitution $y=1-x$ ($x\in ]0, 1[$) inequality \eqref{Eq3.1.1} implies that
\[
\left|f(x)+(1-x)^{\alpha}f(1)-f(1-x)-x^{\alpha}f(1)\right|\leq \varepsilon .
\quad \left(x\in ]0, 1[\right)
\]
From the proof of Theorem \ref{T3.1.1} (see definition \eqref{Eq3.1.27} it is known that
\[
f(x)=f_{0}(x)+c(1-x)^{\alpha} ,
\quad \left(x\in ]0, 1[\right)
\]
therefore the last inequality yields that
\begin{equation}\label{Eq3.1.36}
\left|f_{0}(x)-f_{0}(1-x)+c(1-x)^{\alpha}-cx^{\alpha}+(1-x)^{\alpha}f(1)-x^{\alpha}f(1)\right|\leq \varepsilon
\end{equation}
holds for all $x\in ]0, 1[$.
Whereas
\[
\left|f_{0}(x)-f_{0}(1-x)\right|\leq T(\alpha).
\quad \left(x\in ]0, 1[\right)
\]
Thus after rearranging \eqref{Eq3.1.36}, we get that
\[
\left|f_{0}(x)-f_{0}(1-x)-[c+f(1)][x^{\alpha}-(1-x)^{\alpha}]\right|\leq \varepsilon,
\quad \left(x\in ]0, 1[\right)
\]
that is,
\[
\left|\left|f_{0}(x)-f_{0}(1-x)\right|-\left|c+f(1)\right|\cdot \left|x^{\alpha}-(1-x)^{\alpha}\right|\right|
\leq \varepsilon
\]
holds for all $x\in ]0, 1[$. Therefore
\[
\left|c+f(1)\right|\cdot \left|x^{\alpha}-(1-x)^{\alpha}\right|\leq (T(\alpha)+1)\varepsilon
\]
for all $x\in ]0, 1[$.
Taking the limit $x\to 0+$, we obtain that
\[
\left|c+f(1)\right|\leq (T(\alpha)+1)\varepsilon.
\]
However, in the proof of Theorem \ref{T3.1.1}. we used the definition $c=b-a$, thus
\[
\left|f(1)-(a-b)\right|\leq (T(\alpha)+1)\varepsilon,
\]
so \eqref{Eq3.1.34} holds, indeed.

Finally, we investigate the case $\alpha=0$.
If $x=0$ or $x=1$, then \eqref{Eq3.1.35} trivially holds, since
\[
\left|f(0)-h_{2}(0)\right|=\left|f(0)-f(0)\right|=0\leq K(\alpha)\varepsilon
\]
and
\[
\left|f(1)-h_{2}(1)\right|=\left|f(1)-f(1)\right|=0\leq K(\alpha)\varepsilon .
\]
Let now $x\in ]0, 1[$ and $y=1-x$ in \eqref{Eq3.1.1}, then we obtain that
\begin{equation}\label{Eq3.1.37}
\left|f(x)-f(1-x)\right|\leq \varepsilon,  \quad \left(x\in ]0, 1[\right)
\end{equation}
if fulfilled for all $x\in ]0, 1[$.

Due to Theorem \ref{T3.1.1}. there
exists a logarithmic function $l:]0, 1[\to\mathbb{R}$ and
$c\in\mathbb{R}$ such that
\begin{equation}\label{Eq3.1.38}
\left|f(x)-l(1-x)-c\right|\leq 63\varepsilon
\end{equation}
holds for all $x\in ]0, 1[$. Hence it is enough to prove that the
function $l$ is identically zero on $]0, 1[$.
Indeed, due to \eqref{Eq3.1.37} and \eqref{Eq3.1.38}
\begin{multline}\label{Eq3.1.39}
\left|l(1-x)-l(x)\right| \\
=\left|l(1-x)-f(1-x)+f(1-x)+c
-l(x)+f(x)-f(x)-c\right| \\
\leq
\left|l(1-x)+c-f(x)\right|+\left|f(1-x)-l(x)-c\right|
+\left|f(x)-f(1-x)\right| \\
\leq 127 \varepsilon
\end{multline}
holds for all $x\in ]0, 1[$. Since the function $l$
is uniquely extendable to $\mathbb{R}_{++}$, with the
substitution $x=\frac{p}{p+q}$ ($p, q\in\mathbb{R}$), we get that
\[
\left|l(p)-l(q)\right|\leq 127\varepsilon,  \quad \left(p, q\in \mathbb{R}_{++}\right)
\]
where we used the fact that $l$ is logarithmic, as well.
This last inequality, with the substitution $q=1$ implies that
\[
\left|l(p)\right|\leq 127\varepsilon
\]
holds for all $p\in\mathbb{R}_{++}$, since $l(1)=0$. Thus
$l$ is bounded on $\mathbb{R}_{++}$.
However, the only bounded, logarithmic function on
$\mathbb{R}_{++}$ is the identically zero function.
Therefore,
\[
\left|f(x)-c\right|\leq 63\varepsilon
\]
holds for all $x\in ]0, 1[$, i.e., \eqref{Eq3.1.35} is proved.
\end{proof}

Since
\[
\lim_{\alpha\to 1}K(\alpha)=+\infty,
\]
our method is inappropriate if $\alpha=1$. Hence we cannot prove stability
concerning the fundamental equation of information neither on the set $D^{\circ}$  nor on $D$.

The stability problem for the fundamental equation of information
was raised by L.~Sz\'ekelyhidi, see 38.~Problem in \cite{Sze91}, and
to the best of the authors' knowledge, it is still open. Therefore,
we also can formulate the following.

\begin{Opp}
 Let $\varepsilon\geq 0$ be arbitrary and
$f\colon ]0, 1[\to\mathbb{R}$ be a functions.
Suppose that
\[
 \left|f(x)+(1-x)f\left(\frac{y}{1-x}\right)-f(y)-(1-y)f\left(\frac{x}{1-y}\right)\right|\leq
\]
holds for all $(x, y)\in D^{\circ}$.
Is it true that in this case there exists a solution of the fundamental
equation of information $h\colon ]0, 1[\to\mathbb{R}$ and a constant $K(\varepsilon)\in\mathbb{R}$ depending
only on $\varepsilon$ such that
\[
 \left|f(x)-h(x)\right|\leq K(\varepsilon)
\]
is fulfilled for any $x\in ]0, 1[$?
\end{Opp}

Concerning this problem, we remark that for the
system of recursive, $3$-semi-symmetric information measures, some partial
results are known, see Morando \cite{Mor01}.

Applying Theorem \ref{T3.1.1}. we can prove the stability of a
system of functional equations that characterizes the
$\alpha$-recursive, $3$-semi-symmetric information measures.

\begin{thm}\label{T3.1.3}
 Let $n\geq 2$ be a fixed positive integer
and $(I_{n})$ be the sequence of functions
$I_{n}:\Gamma^{\circ}_{n}\to\mathbb{R}$
and suppose that there exist a sequence $(\varepsilon_{n})$
of nonnegative real numbers and a real number $0\leq\alpha\neq 1$
such that
\begin{multline}\label{Eq3.1.40}
\left|I_{n}(p_{1}, \ldots, p_{n})\right.\\
-I_{n-1}(p_{1}+p_{2}, p_{3}, \ldots, p_{n})-
\left.(p_{1}+p_{2})^{\alpha}I_{2}\left(\frac{p_{1}}{p_{1}+p_{2}}, \frac{p_{2}}{p_{1}+p_{2}}\right)\right|
\leq
\varepsilon_{n-1}
\end{multline}
for all $n\geq 3$ and $(p_{1}, \ldots, p_{n})\in\Gamma^{\circ}_{n}$, and
\begin{equation}\label{Eq3.1.41}
\left|I_{3}(p_{1}, p_{2}, p_{3})-I_{3}(p_{1}, p_{3}, p_{2})\right|\leq \varepsilon_{1}
\end{equation}
holds on $\Gamma^{\circ}_{n}$.
Then, in case $\alpha=0$ there exists a logarithmic function
$l:]0, 1[\to\mathbb{R}$ and $c\in\mathbb{R}$ such that
\begin{multline}\label{Eq3.1.42}
\left|I_{n}\left(p_{1}, \ldots, p_{n}\right)-\left[cH^{0}_{n}\left(p_{1}, \ldots, p_{n}\right)+l(p_{1})\right]\right|\\
\leq \sum^{n-1}_{k=2}\varepsilon_{k}+\left(n-1\right)K(\alpha)\left(2\varepsilon_{2}+\varepsilon_{1}\right)
\end{multline}
for all $n\geq 2$ and $\left(p_{1}, \ldots, p_{n}\right)\in\Gamma^{\circ}_{n}$.
Furthermore, if $\alpha>0$ then
there exist $c, d\in\mathbb{R}$ such that
\begin{multline}\label{Eq3.1.43}
\left|I_{n}(p_{1}, \ldots, p_{n})-
\left[cH^{\alpha}_{n}(p_{1}, \ldots, p_{n})+d(p^{\alpha}_{1}-1)\right]\right|\\
\leq \sum^{n-1}_{k=2}\varepsilon_{k}+(n-1)K(\alpha)(2\varepsilon_{2}+\varepsilon_{1})
\end{multline}
holds for all $n\geq 2$ and $(p_{1}, \ldots, p_{n})\in\Gamma^{\circ}_{n}$, where the convention
$\sum^{1}_{k=2}\varepsilon_{k}=0$ is adopted.
\end{thm}
\begin{proof}
Similarly as in the proof of Theorem \ref{T3.1.3}., due to \eqref{Eq3.1.40} and \eqref{Eq3.1.41},
it can be proved that,
for the function $f$ defined on $]0, 1[$ by
$f(x)=I_{2}(1-x, x)$ we get that
\[
\left|f(x)+(1-x)^{\alpha}f\left(\frac{y}{1-x}\right)
-f(y)-(1-y)^{\alpha}f\left(\frac{x}{1-y}\right)\right|\leq 2\varepsilon_{2}+\varepsilon_{1}
\]
for all $(x, y)\in D^{\circ}$, i.e., \eqref{Eq3.1.1} holds with
$\varepsilon=2\varepsilon_{2}+\varepsilon_{1}$.
Therefore, applying Theorem \ref{T3.1.3}. we obtain
\eqref{Eq3.1.2} and \eqref{Eq3.1.3}, respectively, with some
$a, b, c\in\mathbb{R}$ and a logarithmic function $l:]0, 1[\to\mathbb{R}$
and
$\varepsilon=2\varepsilon_{2}+\varepsilon_{1}$, i.e.,

\[
\left|I_{2}\left(1-x, x\right)-\left(ax^{\alpha}+b(1-x)^{\alpha}-b\right)\right|\leq K(\alpha)(2\varepsilon_{2}+\varepsilon_{1}),
\quad \left(x\in ]0, 1[\right)
\]
in case $\alpha\neq 0$, and
\[
\left|I_{2}\left(1-x, x\right)-\left(l(1-x)+c\right)\right|\leq K(\alpha)(2\varepsilon_{2}+\varepsilon_{1})
\quad \left(x\in ]0, 1[\right)
\]
in case $\alpha=0$.

Therefore \eqref{Eq3.1.42} holds  with $c=(2^{1-\alpha}-1)a$,
$d=b-a$ in case $\alpha>0$
and \eqref{Eq3.1.43} holds in case $\alpha=0$, respectively, for $n=2$.

We continue the proof by induction on $n$.
Suppose that \eqref{Eq3.1.42} and \eqref{Eq3.1.43} hold, resp., and for the
sake of brevity, introduce the notation
\[
J_{n}(p_{1}, \ldots, p_{n})=\left\{
\begin{array}{lcl}
cH^{\alpha}_{n}(p_{1}, \ldots, p_{n}),&\text{if}& \alpha\neq 0 \\
cH^{0}_{n}(p_{1}, \ldots, p_{n})+l(p_{1}),&\text{if}& \alpha=0
\end{array}
\right.
\]
for all $n\geq 2$, $(p_{1}, \ldots, p_{n})\in\Gamma^{\circ}_{n}$.
It can easily be seen that \eqref{Eq3.1.42} and \eqref{Eq3.1.43} hold
on $\Gamma^{\circ}_{n}$ for $J_{n}$ instead of $I_{n}$ ($n\geq 3$)
with $\varepsilon_{n}=0$ ($n\geq 2$).

Therefore, if $\alpha=0$, \eqref{Eq3.1.40} (with $n+1$ instead of $n$),
\eqref{Eq3.1.42} with $n=2$ and the induction hypothesis
(applying to $(p_{1}+p_{2}, \ldots, p_{n+1})$ instead of $(p_{1}, \ldots, p_{n})$)
imply that

\begin{multline*}
\left|I_{n+1}(p_{1}, \ldots, p_{n+1})-J_{n+1}(p_{1}, \ldots, p_{n+1})\right| \\
\leq \varepsilon_{n}+\sum^{n-1}_{k=2}\varepsilon_{k}+K(\alpha)(n-1)(2\varepsilon_{2}+\varepsilon_{1})+
K(\alpha)(2\varepsilon_{2}+\varepsilon_{1}) \\
=\sum^{n}_{k=2}\varepsilon_{k}+K(\alpha)n(2\varepsilon_{2}+\varepsilon_{1}).
\end{multline*}

This yields that \eqref{Eq3.1.42} holds for $n+1$ instead of $n$.

Furthermore, if $\alpha>0$, then \eqref{Eq3.1.40} (with $n+1$ instead of $n$),
\eqref{Eq3.1.43} with $n=2$ and the induction hypothesis
(applying to $(p_{1}+p_{2}, \ldots, p_{n+1})$ instead of $(p_{1}, \ldots, p_{n})$)
imply that

\begin{multline*}
\left|I_{n+1}(p_{1}, \ldots, p_{n+1})-J_{n+1}(p_{1}, \ldots, p_{n+1})\right| \\
\leq \varepsilon_{n}+\sum^{n-1}_{k=2}\varepsilon_{k}+K(\alpha)(n-1)(2\varepsilon_{2}+\varepsilon_{1})+
K(\alpha)(2\varepsilon_{2}+\varepsilon_{1}) \\
=\sum^{n}_{k=2}\varepsilon_{k}+K(\alpha)n(2\varepsilon_{2}+\varepsilon_{1}),
\end{multline*}

that is, \eqref{Eq3.1.43} holds for $n+1$ instead of $n$.
\end{proof}

\subsection{The case $\alpha<0$}

At this part of the paper we will turn to investigate the
case $\alpha<0$. Here it will be proved the for negative parameters,
the parametric fundamental equation of information is
\emph{hyperstable} on $D^{\circ}$ as well as on $D$. As an application of these
results, we will deduce that the system of $\alpha$-recursive, $3$-semi-symmetric
information measures is stable.

\begin{thm}\label{T3.2.1}
Let $\alpha, \varepsilon\in\mathbb{R}$, $\alpha<0$, $\varepsilon\geq 0$
and $f:]0, 1[\to\mathbb{R}$ be a function.
Assume that
\begin{equation}\label{Eq3.2.1}
\left|f(x)+(1-x)^{\alpha}f\left(\frac{y}{1-x}\right)-
f(y)-(1-y)^{\alpha}f\left(\frac{x}{1-y}\right)\right|\leq \varepsilon
\end{equation}
holds for all $(x, y)\in D^{\circ}$.
Then, and only then, there exist $c, d\in\mathbb{R}$ such that
\begin{equation}\label{Eq3.2.2}
f(x)=cx^{\alpha}+d(1-x)^{\alpha}-d
\end{equation}
for all $x\in ]0, 1[$.
\end{thm}

\begin{proof}
It is easy to see that for the function $f$ is given by formula
\eqref{Eq3.2.2} functional equation 
\[
f(x)+(1-x)^{\alpha}f\left(\frac{y}{1-x}\right)=
f(y)+(1-y)^{\alpha}f\left(\frac{x}{1-y}\right)
\]
holds for all $(x, y)\in D^{\circ}$.
Thus inequality \eqref{Eq3.2.1} is also satisfied with arbitrary $\varepsilon\geq 0$.
Therefore it is enough to prove the converse direction.

Define the function $G:D^{\circ}\to\mathbb{R}$ by
\begin{equation}\label{Eq3.2.3}
G(x, y)=f(x)+(1-x)^{\alpha}f\left(\frac{y}{1-x}\right)-f(x+y).
\quad \left((x, y)\in D^{\circ}\right)
\end{equation}
Then inequality \eqref{Eq3.2.1} immediately implies that
\begin{equation}\label{Eq3.2.4}
\left|G(x, y)-G(y, x)\right|\leq \varepsilon
\end{equation}
for all $(x, y)\in D^{\circ}$.

Let $(x, y, z)\in D^{\circ}_{3}$, then due to the definition of the function $G$,
\[
G(x+y, z)=f(x+y)+(1-(x+y))^{\alpha}f\left(\frac{z}{1-(x+y)}\right)-f(x+y+z),
\]

\[
G(x, y+z)=f(x)+(1-x)^{\alpha}f\left(\frac{y+z}{1-x}\right)-f(x+y+z)
\]
and

\begin{multline*}
(1-x)^{\alpha}G\left(\frac{y}{1-x}, \frac{z}{1-x}\right)\\
=(1-x)^{\alpha}\left[f\left(\frac{y}{1-x}\right)+
\left(1-\frac{y}{1-x}\right)^{\alpha}f\left(\frac{\frac{z}{1-x}}{1-\frac{y}{1-x}}\right)-
f\left(\frac{y+z}{1-x}\right)\right],
\end{multline*}
therefore
\begin{equation}\label{Eq3.2.5}
G(x, y)+G(x+y, z)
=G(x, y+z)+(1-x)^{\alpha}G\left(\frac{y}{1-x}, \frac{z}{1-x}\right)
\end{equation}
holds on $D^{\circ}_{3}$, where we used the identity
\[
\dfrac{z}{1-(x+y)}=\dfrac{\dfrac{z}{1-x}}{1-\dfrac{y}{1-x}}
\]
also.

In what follows we will show that the function $G$ is $\alpha$--homogeneous.
Indeed, interchanging $x$ and $y$ in \eqref{Eq3.2.5}, we get
\begin{multline*}
G(y, x)+G(x+y, z)\\
=G(y, x+z)+(1-y)^{\alpha}G\left(\frac{x}{1-y}, \frac{z}{1-y}\right).
\quad \left((x, y, z)\in D^{\circ}_{3}\right)
\end{multline*}
Furthermore, equation \eqref{Eq3.2.5} with the substitution
\[
(x, y, z)=(y, z, x)
\]
yields that
\[
G(y, z)+G(y+z, x)=G(y, x+z)+(1-y)^{\alpha}G\left(\frac{z}{1-y}, \frac{x}{1-y}\right)
\]
is fulfilled for all $(x, y, z)\in D^{\circ}_{3}$.

Thus
\begin{multline}\label{Eq3.2.6}
G(y, z)-(1-x)^{\alpha}G\left(\frac{y}{1-x}, \frac{z}{1-x}\right) \\
=\left\{G(x, y)+G(x+y, z)-G(x, y+z)-(1-x)^{\alpha}G\left(\frac{y}{1-x}, \frac{z}{1-x}\right)\right\} \\
-G(x, y)-G(x+y, z)+G(x, y+z) \\
+\left\{G(y, x)+G(x+y, z)-G(y, x+z)-(1-y)^{\alpha}G\left(\frac{x}{1-y}, \frac{z}{1-y}\right)\right\} \\
+\left\{G(y, z)+G(y+z, x)-G(y, x+z)-(1-y)^{\alpha}G\left(\frac{z}{1-y}, \frac{x}{1-y}\right)\right\} \\
-G(y+z, x)+G(y, x+z)+(1-y)^{\alpha}G\left(\frac{z}{1-y}, \frac{x}{1-y}\right) \\
=G(y, x)-G(x, y) +G(x, y+z)-G(y+z, x) \\
+(1-y)^{\alpha}\left(G\left(\frac{z}{1-y}, \frac{x}{1-y}\right)-G\left(\frac{x}{1-y}, \frac{z}{1-y}\right)\right)
\end{multline}
for all $(x, y, z)\in D^{\circ}_{3}$, since the expressions in the curly brackets are zeros.
Thus \eqref{Eq3.2.6}, \eqref{Eq3.2.4} and the triangle inequality imply that
\begin{equation}\label{Eq3.2.6*}
\left|G(y, z)-(1-x)^{\alpha}G\left(\frac{y}{1-x}, \frac{z}{1-x}\right)\right|
\leq \left(2+(1-y)^{\alpha}\right)\varepsilon
\end{equation}
is fulfilled for all $(x, y, z)\in D^{\circ}_{3}$.
Given any $t\in ]0, 1[$, $(u, v)\in D^{\circ}$, let
\[
x=1-t, \quad y=tu \quad \text{and} \quad z=tv.
\]
Then $x, y, z\in ]0, 1[$ and
\[
x+y+z=1-t(1-u-v)\in ]0, 1[,
\]
that is $(x, y, z)\in D^{\circ}_{3}$, and inequality \eqref{Eq3.2.6*} implies that
\[
\left|G(tu, tv)-t^{\alpha}G(u, v)\right|\leq \left(2+(1-tu)^{\alpha}\right)\varepsilon,
\]
or, after rearranging,
\[
\left|\frac{G(tu, tv)}{t^{\alpha}}-G(u, v)\right|\leq \frac{\left(2+(1-tu)^{\alpha}\right)}{t^{\alpha}}\varepsilon
\]
holds for arbitrary $t\in ]0, 1[$ and $(u, v)\in D^{\circ}$.
Taking the limit $t\to 0+$ we obtain that
\[
\lim_{t\to 0+}\frac{G(tu, tv)}{t^{\alpha}}=G(u, v), \quad \left((u, v)\in D^{\circ}\right)
\]
since $\lim_{t\to 0+}(1-tu)^{\alpha}=1$ for all $u\in ]0, 1[$ and
$\lim_{t\to 0+}t^{-\alpha}=0$, since $\alpha<0$.
This implies that the function $G$ is $\alpha$--homogeneous on $D^{\circ}$.
Indeed, for arbitrary $s\in ]0, 1[$ and $(u, v)\in D^{\circ}$

\begin{multline}\label{Eq3.2.7}
G(su, sv)=\lim_{t\to 0+}\frac{G(t(su), t(sv))}{t^{\alpha}}
\\=s^{\alpha}\lim_{t\to 0+}\frac{G\left((ts)u, (ts)v\right)}{(ts)^{\alpha}}=s^{\alpha}G(u, v).
\end{multline}
At this point of the proof we will show that inequality \eqref{Eq3.2.4} and equation
\eqref{Eq3.2.7} together imply the symmetry of the function $G$.
Indeed, due to \eqref{Eq3.2.4}
\[
\left|G\left(tx, ty\right)-G\left(ty, tx\right)\right|\leq \varepsilon
\]
holds for all $(x, y)\in D^{\circ}$ and $t\in ]0, 1[$.
Using the $\alpha$-homogeneity of the function $G$, we obtain that
\[
\left|t^{\alpha}G\left(x, y\right)-t^{\alpha}G\left(y, x\right)\right|\leq\varepsilon,
\quad \left((x, y)\in D^{\circ}, t\in ]0, 1[\right)
\]
or, if we rearrange this,
\[
\left|G\left(x, y\right)-G\left(y, x\right)\right|\leq \frac{\varepsilon}{t^{\alpha}}
\]
holds for all $(x, y)\in D^{\circ}$ and $t\in ]0, 1[$. Taking the limit $t\to 0+$, we get
that
\[
G(x, y)=G(y, x)
\]
is fulfilled for all $(x, y)\in D^{\circ}$, since $\alpha<0$.
Therefore the function $G$ is symmetric. Due to definition \eqref{Eq3.2.3} this
implies that
\[
f(x)+(1-x)^{\alpha}f\left(\frac{y}{1-x}\right)=
f(y)+(1-y)^{\alpha}f\left(\frac{x}{1-y}\right),  \quad \left((x, y)\in D^{\circ}\right)
\]
i.e., the function $f$ satisfies the parametric fundamental equation of information on
$D^{\circ}$. Thus by Theorem 3. of Maksa \cite{Mak82} there exist $c, d\in\mathbb{R}$
such that
\[
f(x)=cx^{\alpha}+d(1-x)^{\alpha}-d
\]
holds for all $x\in ]0, 1[$.
\end{proof}

In what follows, we will show that for negative $\alpha$'s, the
parametric fundamental equation of information is stable also on the set $D$.

\begin{thm}\label{Thm5.1.2}
Let $\alpha, \varepsilon \in\mathbb{R}$ be fixed, $\alpha<0$, $\varepsilon\geq 0$.
Then the function $f:[0,1]\to\mathbb{R}$ satisfies the inequality
\eqref{Eq3.2.1} for all $(x, y)\in D$ if,
and only if, there exist $c, d\in\mathbb{R}$
such that
\begin{equation}\label{Eq3.2.8}
f(x)=\left\{
\begin{array}{lcl}
0, & \mbox{if} & x=0 \\
cx^{\alpha}+d\left(1-x\right)^{\alpha}-d, & \mbox{if} & x\in ]0,1[ \\
c-d, & \mbox{if} & x=1.
\end{array}
\right.
\end{equation}
\end{thm}
\begin{proof}
Let $y=0$ in \eqref{Eq3.2.1}. Then we have that
\[
\left(\left(1-x\right)^{\alpha}+1\right)\left|f\left(0\right)\right|\leq \varepsilon
\quad \left(x\in ]0,1[\right)
\]
Since $\alpha<0$, this yields that $f\left(0\right)=0$.
On the other hand, by Theorem \ref{T3.2.1},
\[
f\left(x\right)=c x^{\alpha}+d\left(1-x\right)^{\alpha}-d
\quad \left(x\in ]0, 1[\right)
\]
with some $c, d\in\mathbb{R}$. Finally, let $x\in ]0, 1[$ and $y=1-x$ in
\eqref{Eq3.2.1}. Then, again by Theorem \ref{T3.2.1}., there exist $c, d\in \mathbb{R}$
such that
\[
\left|c-d-f\left(1\right)\right|
\left|x^{\alpha}-\left(1-x\right)^{\alpha}\right|\leq
\varepsilon.
\]
Since $\alpha<0$, $f\left(1\right)=c-d$ follows. \\
The converse is an easy computation and it turns out that
$f$ defined by \eqref{Eq3.2.8} is a solution of \eqref{Eq1.1.6} on $D$.
\end{proof}

Our third main result in this section says that the system of
$\alpha$-recursive, $3$-semi-symmetric information measures is stable.

\begin{thm}\label{T3.2.3}
Let $n\geq 2$ be a fixed positive integer,
$\left(I_{n}\right)$ be the sequence of functions
$I_{n}:\Gamma^{\circ}_{n}\to\mathbb{R}$ and suppose that there exist a sequence
$\left(\varepsilon_{n}\right)$ of nonnegative real numbers and a real number
$\alpha<0$ such that
\begin{multline}\label{Eq3.2.9}
\left|I_{n}\left(p_{1}, \ldots, p_{n}\right)-\right. \\
\left. I_{n-1}\left(p_{1}+p_{2}, p_{3}, \ldots, p_{n}\right)-
\left(p_{1}+p_{2}\right)^{\alpha}I_{2}\left(\frac{p_{1}}{p_{1}+p_{2}},
\frac{p_{2}}{p_{1}+p_{2}}\right)\right|
\leq \varepsilon_{n-1}
\end{multline}
holds for all $n\geq 3$ and $\left(p_{1}, \ldots, p_{n}\right)\in\Gamma^{\circ}_{n}$,
and
\begin{equation}\label{Eq3.2.10}
\left|I_{3}\left(p_{1}, p_{2}, p_{3}\right)-I_{3}\left(p_{1}, p_{3}, p_{2}\right)\right|\leq
\varepsilon,
\end{equation}
holds on $D^{\circ}_{3}$.
Then there exist $a, b\in\mathbb{R}$ such that
\begin{equation}\label{Eq3.2.11}
\left|
I_{n}\left(p_{1}, \ldots, p_{n}\right)-
\left(a H^{\alpha}_{n}\left(p_{1}, \ldots, p_{n}\right)+b\left(p^{\alpha}_{1}-1\right)\right)\right|\leq \sum^{n-1}_{k=2}\varepsilon_{k}
\end{equation}
for all $n\geq 2$ and $\left(p_{1}, \ldots, p_{n}\right)\in\Gamma^{\circ}_{n}$,
where the convention $\sum^{1}_{k=2}\varepsilon_{k}=0$ is adopted.
\end{thm}

\begin{proof}
As in Maksa \cite{Mak08}, it can be proved that, due to
\eqref{Eq3.2.9} and \eqref{Eq3.2.10}, for the function $f$
defined by $f(x)=I_{2}\left(1-x, x\right)$, $x\in ]0,1[$ we get that
\[
\left|
f\left(x\right)+\left(1-x\right)^{\alpha}f\left(\frac{y}{1-x}\right)-
f\left(y\right)-\left(1-y\right)^{\alpha}f\left(\frac{x}{1-y}\right)\right|
\leq 2\varepsilon_{2}+\varepsilon_{1}
\]
for all $\left(x, y\right)\in D^{\circ}$, i.e., \eqref{Eq3.2.1}
holds with $\varepsilon=2\varepsilon_{2}+\varepsilon_{1}$.
Therefore, applying Theorem \ref{T3.2.1}., we obtain
\eqref{Eq3.2.2} with some $c, d\in\mathbb{R}$, i.e.,
\[
I_{2}\left(1-x, x\right)=cx^{\alpha}+ d \left(1-x\right)^{\alpha}-d,
\quad \left(x\in ]0,1[\right)
\]
i.e., \eqref{Eq3.2.11} holds for $n=2$ with $a=(2^{1-\alpha}-1)c$, $b=d-c$.

We continue the proof by induction on $n$.
Suppose that \eqref{Eq3.2.11} holds and, for the sake of brevity, introduce the
notation
\[
J_{n}\left(p_{1}, \ldots, p_{n}\right)=
a H^{\alpha}_{n}\left(p_{1}, \ldots, p_{n}\right)+b\left(p^{\alpha}_{1}-1\right)
\]
for all $n\geq 2$, $\left(p_{1}, \ldots, p_{n}\right)\in\Gamma^{\circ}_{n}$.
It can easily be seen that \eqref{Eq3.2.9} and \eqref{Eq3.2.10} hold on
$\Gamma^{\circ}_{n}$ for $J_{n}$ instead of $I_{n}$ $(n\geq 3)$ with
$\varepsilon_{n}=0$ $(n\geq 2)$.
Thus for all $\left(p_{1}, \ldots, p_{n+1}\right)\in\Gamma^{\circ}_{n+1}$,
we get that

\begin{multline*}
I_{n+1}\left(p_{1}, \ldots, p_{n+1}\right)-J_{n+1}\left(p_{1}, \ldots, p_{n+1}\right) \\
= I_{n+1}\left(p_{1}, \ldots, p_{n+1}\right)-
J_{n}\left(p_{1}+p_{2}, p_{3}, \ldots, p_{n+1}\right)\\-
\left(p_{1}+p_{2}\right)^{\alpha}
J_{2}\left(\frac{p_{1}}{p_{1}+p_{2}}, \frac{p_{2}}{p_{1}+p_{2}}\right) \\
= I_{n+1}\left(p_{1}, \ldots, p_{n+1}\right)-
I_{n}\left(p_{1}+p_{2}, p_{3}, \ldots, p_{n+1}\right)\\-
\left(p_{1}+p_{2}\right)^{\alpha}
I_{2}\left(\frac{p_{1}}{p_{1}+p_{2}}, \frac{p_{2}}{p_{1}+p_{2}}\right) \\
+ I_{n}\left(p_{1}+p_{2}, p_{3}, \ldots, p_{n+1}\right)-
J_{n}\left(p_{1}+p_{2}, p_{3}, \ldots, p_{n+1}\right) \\
+ \left(p_{1}+p_{2}\right)^{\alpha}
I_{2}\left(\frac{p_{1}}{p_{1}+p_{2}}, \frac{p_{2}}{p_{1}+p_{2}}\right)\\-
\left(p_{1}+p_{2}\right)^{\alpha}
J_{2}\left(\frac{p_{1}}{p_{1}+p_{2}}, \frac{p_{2}}{p_{1}+p_{2}}\right).
\end{multline*}

Therefore \eqref{Eq3.2.11} with $n=2$ and the induction hypothesis imply that
\[
\left|
I_{n+1}\left(p_{1}, \ldots, p_{n+1}\right)-
J_{n}\left(p_{1}, \ldots, p_{n+1}\right)\right|\leq
\varepsilon_{n}+\sum^{n-1}_{k=2}\varepsilon_{k}=
\sum^{n}_{k=2}\varepsilon_{k},
\]
that is, \eqref{Eq3.2.11} holds for $n+1$ instead of $n$.
\end{proof}

\begin{cor}
Applying Theorem \ref{T3.2.3}. with the choice $\varepsilon_{n}=0$
for all $n\in\mathbb{N}$, we get the $\alpha$-recursive, $3$-semi-symmetric
information measures. Hence the previous theorem says that the system of
$\alpha$-recursive and $3$-semi-symmetric information measures is stable.
\end{cor}

\subsection{Related equations}

In the previous subsections we have investigated the stability problem
of the parametric fundamental equation of information.
In the remaining part of our paper, we will discuss the stability problem
of some functional equation that also have information theoretical background.
Firstly, we will show that the so-called \emph{entropy equation} is stable
on it domain. After that some results concerning the
\emph{modified entropy equation} will follow.
Finally, we will end this section with some open problems.

\subsubsection{Stability of the entropy equation}

In what follows, our aim  is to prove that \emph{the entropy equation}, i.e., equation
\begin{equation}\label{Eq3.3.1}
H\left(x, y, z\right)=H\left(x+y, 0, z\right)+H\left(x, y, 0\right)
\end{equation}
is stable on the set
\[
 C=\left\{(x, y, z)\in\mathbb{R}^{3}\vert x\geq 0, y\geq 0, z\geq 0, x+y+z>0\right\}.
\]

In \cite{KamMik74} A.~Kami\'{n}ski and J.~Mikusi\'{n}ski
determined the continuous and $1$-homoge\-neous solutions of
equation (\ref{Eq3.3.1}) on the set $\mathbb{R}^{3}$.
This result was strengthened by
J.~Acz\'{e}l in \cite{Acz77}. After that, using a result of Jessen--Karpf--Thorup \cite{JesKarTho68},
which concerns the solution of
the cocycle equation, Z.~Dar\'{o}czy proved the following (see \cite{Dar76}).

\begin{thm}\label{T2.3.1}
If a function $H:C\to\mathbb{R}$ is symmetric in $C$ and
satisfies the equation \eqref{Eq3.3.1} in the interior of $C$ and
the map $(x, y)\mapsto H\left(x, y, 0\right)$ is positively
homogeneous (of order $1$) for all $x, y\in\mathbb{R}_{++}$,
then there exists a function
$\varphi:\mathbb{R}_{++}\to\mathbb{R}$ such that
\[
\varphi\left(xy\right)=x\varphi\left(y\right)+y\varphi\left(x\right)
\]
holds for all $x, y \in\mathbb{R}_{++}$ and
\[
H\left(x, y, z\right)=\varphi\left(x+y+z\right)-\varphi(x)-\varphi(y)-\varphi(z)
\]
for all $\left(x, y, z\right)\in C$.
\end{thm}

During the proof of the main result the stability of the \emph{cocycle equation} is needed.
This theorem can be found in \cite{Sze95}.

\begin{thm}\label{T3.3.2}
Let $S$ be a right amenable semigroup and let
$F:S\times S\to\mathbb{C}$ be a function, for which the function
\begin{equation}\label{Eq3.3.2}
\left(x, y, z\right)\longmapsto
F\left(x, y\right)+F\left(x+y, z\right)-F\left(x, y+z\right)-F\left(y, z\right)
\end{equation}
is bounded on $S\times S\times S$.
Then there exists a function
$\Psi:S\times S\to\mathbb{C}$ satisfying the cocycle equation,
i.e.,
\begin{equation}\label{Eq3.3.3}
\Psi\left(x, y\right)+\Psi\left(x+y, z\right)=\Psi\left(x, y+z\right)+\Psi\left(y, z\right)
\end{equation}
for all $x, y, z \in S$ and for which the function $F-\Psi$ is bounded by the same constant as
the map defined by \eqref{Eq3.3.2}.
\end{thm}

About the symmetric, $1$--homogeneous solutions of the cocycle equation  one
can read in \cite{JesKarTho68}.
Furthermore, the symmetric and $\alpha$--homogeneous solutions of equation
\eqref{Eq3.3.3} can be found in \cite{Mak82}, as a consequence of Theorem 3.
The general solution of the cocycle equation without symmetry and homogeneity assumptions,
on cancellative abelian semigroups
was determined by M.~Hossz\'{u} in \cite{Hos71}.

Our main result concerning the stability of equation \eqref{Eq3.3.1} is the following, see also \cite{Gse10}.
\begin{thm}\label{T3.3.3}
Let $\varepsilon_{1}, \varepsilon_{2}, \varepsilon_{3}$ be arbitrary
nonnegative real numbers, $\alpha\in\mathbb{R}$, and assume that the function
$H:C\to\mathbb{R}$ satisfies the following system of inequalities.
\begin{equation}\label{Eq3.3.4}
\left|H(x, y, z)-H\left(\sigma(x), \sigma(y), \sigma(z)\right)\right|\leq \varepsilon_{1}
\end{equation}
for all $(x, y, z)\in C$ and for all
$\sigma:\left\{x, y, z\right\}\mapsto\left\{x, y, z\right\}$
permutation;
\begin{equation}\label{Eq3.3.5}
\left|H\left(x, y, z\right)-H\left(x+y, 0, z\right)-H\left(x, y, 0\right)\right|\leq \varepsilon_{2}
\end{equation}
for all $(x, y, z)\in C^{\circ}$, where $C^{\circ}$ denotes the interior of the set $C$;
\begin{equation}\label{Eq3.3.6}
\left|H\left(tx, ty, 0\right)-t^{\alpha}H(x, y, 0)\right|\leq \varepsilon_{3}
\end{equation}
holds for all $t, x, y\in\mathbb{R}_{++}$.
Then, in case $\alpha=1$ there exists a function $\varphi:\mathbb{R}_{++}\to\mathbb{R}$ which satisfies
the functional equation
\[
\varphi\left(xy\right)=x\varphi\left(y\right)+y\varphi\left(x\right),
\quad \left(x, y\in\mathbb{R}_{++}\right)
\]
and
\begin{equation}\label{Eq3.3.7}
\left|H\left(x, y, z\right)-
\left[\varphi\left(x+y+z\right)-\varphi\left(x\right)-\varphi\left(y\right)-\varphi\left(z\right)\right]\right|
\leq \varepsilon_{1}+\varepsilon_{2}
\end{equation}
holds for all $\left(x, y, z\right)\in C^{\circ}$;
in case $\alpha=0$ there exists a constant $a\in\mathbb{R}$ such that
\begin{equation}\label{Eq3.3.8}
\left|H\left(x, y, z\right)-a\right|\leq
8\varepsilon_{3}+25\varepsilon_{2}+49\varepsilon_{1}
\end{equation}
for all $\left(x, y, z\right)\in C^{\circ}$;
finally, in all other cases there exists a constant $c\in\mathbb{R}$ such that
\begin{equation}\label{Eq3.3.9}
\left|H\left(x, y, z\right)-c\left[\left(x+y+z\right)^{\alpha}-x^{\alpha}-y^{\alpha}-z^{\alpha}\right]\right|
\leq \varepsilon_{1}+\varepsilon_{2}
\end{equation}
holds on $C^{\circ}$.
\end{thm}

\begin{proof}
For the sake of brevity, here we present only the sketch of proof of the above statemant. 
For details, the reader should consult \cite{Gse10}. 

Using inequality \eqref{Eq3.3.6} it can be shown that the map
\[
 (x, y)\mapsto H\left(x, y, 0\right) 
\qquad 
\left(x, y\in\mathbb{R}_{++}\right)
\]
is homogeneous of degree $\alpha$, assuming that $\alpha\neq 0$. 

Let us consider the function $F\colon \mathbb{R}^{2}_{++}\to\mathbb{R}$ defined by 
\[
 F(x, y)=H(x, y, 0)
\qquad 
\left(x, y\in\mathbb{R}_{++}\right). 
\]
From inequalities \eqref{Eq3.3.4} and \eqref{Eq3.3.5} we can deduce that 
\begin{equation}\label{Eq3.3.16}
\left|F(x, y)-F(y, x)\right|\leq \varepsilon_{1}, \quad \left(x, y\in \mathbb{R}_{++}\right)
\end{equation}
and 
\begin{equation}\label{Eq3.3.17}
\left|F(x+y, z)+
F(x, y)-
F(x, y+z)-
F(y, z)\right|\leq 2\varepsilon_{2}+4\varepsilon_{1}.
\quad (x, y, z\in\mathbb{R}_{++})
\end{equation}. 
Furthermore, in case $\alpha\neq 0$, $H(x, y, 0)$ is
homogeneous of degree $\alpha$, therefore
\begin{equation}\label{Eq3.3.18}
F(tx, ty)=t^{\alpha}F(x, y) \quad \left(\alpha\neq 0, t, x, y \in\mathbb{R}_{++}\right)
\end{equation}
and if $\alpha=0$, 
\begin{equation}\label{Eq3.3.19}
\left|F(tx, ty)-F(x, y)\right|\leq \varepsilon_{3},  \quad \left(t, x, y\in\mathbb{R}_{++}\right)
\end{equation}
is fulfilled. 

The set $C^{\circ}$ is a commutative semigroup with the usual addition.
Thus it is amenable, as well.
Therefore, by Theorem \ref{T3.3.2}., there exists a function
$G:\mathbb{R}^{2}_{++}\to\mathbb{R}$ which is a solution of the cocycle
equation, and for which
\begin{equation}\label{Eq3.3.20}
\left|F\left(x, y\right)-G\left(x, y\right)\right|
\leq 2\varepsilon_{2}+4\varepsilon_{1}
\end{equation}
holds for all $x, y\in\mathbb{R}_{++}$.
Additionally, by a result of
\cite{Hos71} there exist a function $f:\mathbb{R}_{++}\to\mathbb{R}$ and a function
$B:\mathbb{R}^{2}_{++}\to\mathbb{R}$ which satisfies the following system
\[
\begin{array}{rcl}
B(x+y, z)&=&B(x, z)+B(y, z), \\
B(x, y)+B(y, x)&=&0,
\end{array}
\quad (x, y, z\in\mathbb{R}_{++})
\]
such that
\[
G\left(x, y\right)=B\left(x, y\right)+f\left(x+y\right)-f\left(x\right)-f\left(y\right).
\quad \left(x, y\in\mathbb{R}_{++}\right)
\]
All in all, this means that
\begin{equation}\label{Eq3.3.21}
\left|F(x, y)-\left(B(x, y)+f(x+y)-f(x)-f(y)\right)\right|\leq 2\varepsilon_{2}+4\varepsilon_{1}
\end{equation}
holds for all $x, y\in\mathbb{R}_{++}$.

Using the above properties of the function $B$, we
can show that $B$ is identically zero on $\mathbb{R}^{2}_{++}$.
Additionally, after some computation, we obtain that 
\[
F(x+y, z)+F(x, y)=F(x, y+z)+F(y, z). \quad \left(x, y, z\in\mathbb{R}_{++}\right)
\]
This means that also the function $F$ satisfies the cocycle equation on $\mathbb{R}^{2}_{++}$.
Additionally, $F$ is homogeneous of degree $\alpha$ ($\alpha\neq 0$) and symmetric.
Using Theorem 5. in \cite{JesKarTho68}, in case
$\alpha=1$, and a result of \cite{Mak82} in all other cases,
we get that
\begin{equation}\label{Eq3.3.23}
F(x, y)=\left\{
\begin{array}{lcl}
c\left[(x+y)^{\alpha}-x^{\alpha}-y^{\alpha}\right], & \hbox{if} & \alpha\notin \left\{0,1\right\} \\
\varphi\left(x+y\right)-\varphi(x)-\varphi(y), & \hbox{if}& \alpha=1
\end{array}
\right.
\end{equation}
where the function $\varphi:\mathbb{R}_{++}\to\mathbb{R}$ satisfies the functional equation
\[
\varphi\left(xy\right)=x\varphi(y)+y\varphi(x)
\]
for all $x, y\in\mathbb{R}_{++}$, and $c\in\mathbb{R}$ is a constant.
In view of the definition of the function $F$, this yields that
\begin{equation}\label{Eq3.3.24}
H(x, y, 0)= c\left[(x+y)^{\alpha}-x^{\alpha}-y^{\alpha}\right]
\end{equation}
for all $x, y\in\mathbb{R}_{++}$ in case $\alpha\notin\left\{0, 1\right\}$, and
\begin{equation}\label{Eq3.3.25}
H(x, y, 0)=\varphi\left(x+y\right)-\varphi(x)-\varphi(y)
\end{equation}
for all $x, y\in\mathbb{R}_{++}$ in case $\alpha=1$.

Using this representations and inequalities \eqref{Eq3.3.4} and \eqref{Eq3.3.5}, the 
statement of our theorem can be deduceed. 
\end{proof}

With the choice $\varepsilon_{1}=\varepsilon_{2}=\varepsilon_{3}=0$ one can
recognize the solutions
of equation \eqref{Eq3.3.1}.

\begin{cor}
Assume that the function
$H:C\to\mathbb{R}$ is symmetric, homogeneous of degree $\alpha$, where
$\alpha\in\mathbb{R}$ is arbitrary but fixed. Furthermore, suppose that
$H$ satisfies equation \eqref{Eq3.3.1} on the set $C^{\circ}$.
Then, in case $\alpha=1$ there exists a function $\varphi:\mathbb{R}_{++}\to\mathbb{R}$ which satisfies
the functional equation
\[
\varphi\left(xy\right)=x\varphi\left(y\right)+y\varphi\left(x\right),
\quad \left(x, y\in\mathbb{R}_{++}\right)
\]
and
\begin{equation}
H\left(x, y, z\right)=
\varphi\left(x+y+z\right)-\varphi\left(x\right)-\varphi\left(y\right)-\varphi\left(z\right)
\end{equation}
holds for all $\left(x, y, z\right)\in C^{\circ}$;
in all other cases there exists a constant $c\in\mathbb{R}$ such that
\begin{equation}
H\left(x, y, z\right)=c\left[\left(x+y+z\right)^{\alpha}-x^{\alpha}-y^{\alpha}-z^{\alpha}\right]
\end{equation}
holds on $C^{\circ}$.
\end{cor}

\begin{rem}
Our theorem says that the entropy equation is stable in the sense of Hyers and Ulam.
\end{rem}

\subsubsection{Stability of the modified entropy equation}

In this part of the paper we investigate the stability problem concerning
the functional equation
\begin{equation}\label{Eq3.3.33}
f(x, y, z)=f(x, y+z, 0)+(y+z)^{\alpha}f\left(0, \frac{y}{y+z}, \frac{z}{y+z}\right),
\end{equation}
where $x, y, z$ are positive real numbers and $\alpha$ is a given real number.
Equation \eqref{Eq3.3.33} is a special case of the
so-called \emph{modified entropy equation},
\begin{equation}\label{Eq3.3.34}
f(x, y, z)=f(x, y+z, 0)+\mu(y+z)f\left(0, \frac{y}{y+z}, \frac{z}{y+z}\right),
\end{equation}
where $\mu$ is a given multiplicative function defined on the positive cone of $\mathbb{R}^{k}$ and
\eqref{Eq3.3.34} is supposed to hold for all elements $x, y, z$ of the above mentioned cone
and all operations on vectors are to be understood componentwise.
The symmetric solutions of equation \eqref{Eq3.3.34} were determined in \cite{Gse08}
(see also \cite{AbbGseMakSun08}).

By a real interval we always mean a subinterval of positive length of $\mathbb{R}$.
Furthermore, in case $U$ and $V$ are real intervals, then their sum
\[
U+V=\left\{u+v \mid u\in U, v\in V\right\}
\]
is obviously a real interval, as well.

During the proof of our main result of this subsection the stability of a simple associativity
equation should be used which is contained in the following theorem, see \cite{Gse10a}.

\begin{thm}\label{T3.3.4}
Let $U, V, W$ be real intervals,
$A:(U+V)\times W\rightarrow\mathbb{R}$,
$B:U\times (V+W)\rightarrow\mathbb{R}$
and suppose that
\begin{equation}\label{Eq3.3.35}
\left|A(u+v, w)-B(u, v+w)\right|\leq \varepsilon
\end{equation}
holds for all $u\in U$, $v\in V$ and $w\in W$.
Then there exists a function $\varphi:U+V+W\rightarrow\mathbb{R}$ such that
\begin{equation}\label{Eq3.3.36}
\left|A(p, q)-\varphi(p+q)\right|\leq 2\varepsilon \quad \left(p\in (U+V), q\in W\right)
\end{equation}
and
\begin{equation}\label{Eq3.3.37}
\left|B(t, s)-\varphi(t+s)\right|\leq \varepsilon \quad \left(t\in U, s\in (V+W)\right)
\end{equation}
hold.
\end{thm}

With the choice $\varepsilon_{1}=\varepsilon_{2}=0$, we get the following theorem.
Nevertheless, it was proved in Maksa \cite{Mak00}.
\begin{cor}
Let $U, V$ and $W$ be real intervals,
$A:(U+V)\times W\rightarrow\mathbb{R}$,
$B:U\times (V+W)\rightarrow\mathbb{R}$ and suppose that
\[
A(u+v, w)=B(u, v+w)
\]
holds for all $u\in U, v\in V$ and $w\in W$.
Then there exists a function $\varphi:U+V+W\rightarrow\mathbb{R}$
such that
\begin{equation}
A(p, q)=\varphi(p+q)
\end{equation}
for all $p\in U+V$ and $q\in W$ and
\begin{equation}
B(t, s)=\varphi(t+s)
\end{equation}
for all $t\in U$ and $s\in V+W$.
\end{cor}

In view of the results of the previous sections (that is Theorems \ref{T3.1.1} and \ref{T3.2.1}) 
and with the help of Theorem \ref{T3.3.4}, the following result can be proved. 
For the details of the proof see \cite{Gse10a}.

\begin{thm}\label{T3.3.5}
Let $\alpha, \varepsilon\in\mathbb{R}$, $\alpha\neq 1, \varepsilon\geq 0$
and $f:\mathbb{R}^{3}_{+}\rightarrow\mathbb{R}$ be a function.
Assume that
\begin{equation}\label{Eq3.3.40}
\left|f(x, y, z)-f(x, y+z, 0)-(y+z)^{\alpha}f\left(0, \frac{y}{y+z}, \frac{z}{y+z}\right)\right|\leq \varepsilon_{1}
\end{equation}
and
\begin{equation}\label{Eq3.3.41}
\left|f(x, y, z)-f\left(\sigma(x), \sigma(y), \sigma(z)\right)\right|\leq \varepsilon_{2}
\end{equation}
hold for all $x, y, z\in\mathbb{R}_{++}$ and for all permutations
$\sigma:\left\{x, y, z\right\}\rightarrow\left\{x, y, z\right\}$.

Then, in case $\alpha<0$, there exist $a\in\mathbb{R}$ and a function
$\varphi_{1}:\mathbb{R}_{++}\rightarrow\mathbb{R}$ such that
\begin{equation}\label{Eq3.3.42}
\left|f(x, y, z)-\left[ax^{\alpha}+ay^{\alpha}+az^{\alpha}+\varphi_{1}(x+y+z)\right]\right|\leq
2\varepsilon_{1}+3\varepsilon_{2}
\end{equation}
holds for all $x, y, z\in\mathbb{R}_{++}$.

Furthermore, if $\alpha=0$, then there exists a function $\varphi_{2}:\mathbb{R}_{++}\rightarrow\mathbb{R}$
such that
\begin{equation}\label{Eq3.3.43}
\left|f(x, y, z)-\varphi_{2}(x+y+z)\right|\leq 191\varepsilon_{1}+1263\varepsilon_{2}
\end{equation}
holds for all $x, y, z\in\mathbb{R}_{++}$.

Finally, if $1\neq \alpha>0$, then for all $n\in\mathbb{N}$, there exists a function
$\psi_{n}:]0, 3n]\rightarrow\mathbb{R}$ such that
\[
\left|f(x, y, z)-\left[ax^{\alpha}+ay^{\alpha}+az^{\alpha}+\psi_{n}(x+y+z)\right]\right|
\leq c_{n}(\alpha)\varepsilon_{n}+d_{n}(\alpha)\varepsilon_{2}
\]
holds for all $x, y, z\in ]0, n]$, where
\[
c_{n}(\alpha)=2+7\cdot 2^{\alpha}n^{\alpha}K(\alpha)
\quad
\text{and}
\quad
d_{n}(\alpha)=4+7\cdot 2^{\alpha+2}n^{\alpha}K(\alpha).
\]
\end{thm}

With the choice $\varepsilon_{1}=\varepsilon_{2}=0$ we get the general
solutions of equation \eqref{Eq3.3.33}, in the investigated cases.

\begin{cor}\label{C3.3.3}
Let $\alpha\in\mathbb{R}$, $\alpha\neq 1$ and
suppose that the function $f:\mathbb{R}^{3}_{+}\rightarrow\mathbb{R}$
is symmetric and satisfies
functional equation \eqref{Eq3.3.33} for all $x, y, z\in\mathbb{R}_{++}$.

Then, in case $\alpha\neq 0$, there exist $a\in\mathbb{R}$ and a function
$\varphi_{1}:\mathbb{R}_{++}\rightarrow\mathbb{R}$ such that
\[
f(x, y, z)=ax^{\alpha}+ay^{\alpha}+az^{\alpha}+\varphi_{1}(x+y+z)
\]
holds for all $x, y, z\in\mathbb{R}_{++}$.

In case $\alpha=0$, there exists a function $\varphi_{2}:\mathbb{R}_{++}\rightarrow\mathbb{R}$
such that
\[
f(x, y, z)=\varphi_{2}(x+y+z)
\]
is fulfilled for all $x, y, z\in\mathbb{R}_{++}$.
\end{cor}

In view of Corollary \ref{C3.3.3}., our theorem says that the modified entropy equation
is stable in the sense of Hyers and Ulam on its one-dimensional domain with the
multiplicative function $\mu(x)=x^{\alpha}$ ($\alpha\leq 0, x\in\mathbb{R}_{++}$).

In case $1\neq \alpha>0$ we obtain however that functional equation
\eqref{Eq3.3.33} is stable on every cartesian product of bounded real
intervals of the form $]0, n]^{3}$ , where $n\in\mathbb{N}$.
Nevertheless, an easy computation shows that
\[
\lim_{n\rightarrow +\infty}c_{n}(\alpha)=+\infty
\quad
\lim_{n\rightarrow +\infty}d_{n}(\alpha)=+\infty.
\quad
\left(1\neq \alpha>0\right)
\]

To the best of our knowledge, this is a new phenomenon in the
stability theory of functional equations.
Since we cannot prove the 'standard' Hyers--Ulam stability in this case,
the following problem can be raised.

\begin{Opp}
Let $\alpha, \varepsilon_{1}, \varepsilon_{2}\in\mathbb{R}$,
$\alpha>0, \varepsilon_{1}, \varepsilon_{2}\geq 0$, and
$f:\mathbb{R}^{3}_{+}\rightarrow\mathbb{R}$ be a function.
Assume that
\[
\left|f(x, y, z)-f(x, y+z, 0)-(y+z)^{\alpha}f\left(0, \frac{y}{y+z}, \frac{z}{y+z}\right)\right|\leq \varepsilon_{1}
\]
and
\[
\left|f(x, y, z)-f\left(\sigma(x), \sigma(y), \sigma(z)\right)\right|\leq \varepsilon_{2}
\]
holds for all $x, y, z\in\mathbb{R}_{++}$ and for all
$\sigma:\left\{x, y, z\right\}\rightarrow\left\{x, y, z\right\}$ permutations.

Is is true that there exists a solution of equation
\eqref{Eq3.3.33} $h:\mathbb{R}^{3}_{++}\rightarrow\mathbb{R}$ such that
\[
\left|f(x, y, z)-h(x, y, z)\right|\leq K_{1}\varepsilon_{1}+K_{2}\varepsilon_{2}
\]
holds for all $x, y, z\in\mathbb{R}_{++}$ with some $K_{1}, K_{2}\in\mathbb{R}$?
\end{Opp}

The second open problem that can be raised is the stability problem of
the modified entropy equation itself, i.e.,
equation \eqref{Eq3.3.34}.

\begin{Opp}
Let $\varepsilon_{1}, \varepsilon_{2}\geq 0$, $\mu:\mathbb{R}^{k}_{++}\rightarrow\mathbb{R}$
be a given multiplicative function, $f:\mathbb{R}^{3k}_{+}\rightarrow\mathbb{R}$.
Assume that
\[
\left|f(x, y, z)-f(x, y+z, 0)-\mu(y+z)f\left(0, \frac{y}{y+z}, \frac{z}{y+z}\right)\right|\leq \varepsilon_{1}
\]
and
\[
\left|f(x, y, z)-f\left(\sigma(x), \sigma(y), \sigma(z)\right)\right|\leq \varepsilon_{2}
\]
holds for all $x, y, z\in\mathbb{R}^{k}_{++}$ and for all
$\sigma:\left\{x, y, z\right\}\rightarrow\left\{x, y, z\right\}$ permutation.

Is is true that there exists a solution of equation
\eqref{Eq3.3.34} $h:\mathbb{R}^{3k}_{++}\rightarrow\mathbb{R}$ such that
\[
\left|f(x, y, z)-h(x, y, z)\right|\leq K_{1}\varepsilon_{1}+K_{2}\varepsilon_{2}
\]
holds for all $x, y, z\in\mathbb{R}^{k}_{++}$ with certain $K_{1}, K_{2}\in\mathbb{R}$?
\end{Opp}

\subsection{Stability of sum form equations}

We have to begin with an open problem since there is no stability result on equation \eqref{Eq1.2.2}
\begin{equation*}
\sum_{i=1}^{n}\sum_{j=1}^{m}f(p_{i}q_{j})=\sum_{i=1}^{n}f(p_{i})+\sum_{j=1}^{m}f(q_{j})+(2^{1-\alpha}-1)\sum_{i=1}^{n}f(p_{i})\sum_{j=1}^{m}f(q_{j})
\end{equation*}
in case $\alpha=1$.

\begin{Opp}
Suppose that $n\geq 2, m\geq 2, 0\leq\varepsilon\in\mathbb{R}, f:I\to\mathbb{R}$ and the stability inequality
\[
\left|\sum_{i=1}^{n}\sum_{j=1}^{m}f(p_{i}q_{j})-\sum_{i=1}^{n}f(p_{i})-\sum_{j=1}^{m}f(q_{j})  \right|\leq\varepsilon
\]
holds for all $(p_{1}, \ldots, p_{n})\in\mathcal{G}_{n}, (q_{1}, \ldots, q_{n})\in\mathcal{G}_{m}$. Prove or disprove that
$f$ is the sum of a solution of \eqref{Eq1.2.2} with $\alpha=1$ and a bounded function.
\end{Opp}

A somewhat related result however is proved in Kocsis--Maksa \cite{KocMak98} which reads as follows.
\begin{thmszn}
Let $n\geq 3, m\geq 3,\,\, 0\leq\varepsilon\in\mathbb{R},\,\, f:[0,1]\to\mathbb{R},\,\, \alpha,\, \beta\in\mathbb{R}$ and suppose that
\[
\left|\sum_{i=1}^{n}\sum_{j=1}^{m}f(p_{i}q_{j})-\sum_{i=1}^{n}f(p_{i})\sum_{j=1}^{m}q_{j}^{\beta}-\sum_{j=1}^{m}f(q_{j})\sum_{i=1}^{n}p_{i}^{\alpha} \right|\leq\varepsilon
\]
holds for all $(p_{1}, \ldots, p_{n})\in\Gamma_{n},\, (q_{1}, \ldots, q_{n})\in\Gamma_{m}$.

Then there exist an additive function $a:\mathbb{R}\to\mathbb{R}$, a function $\ell:\mathbb{R}_{+}\to\mathbb{R},\, \ell(0)=0, \,\ell$ is logarithmic on $\mathbb{R}_{++}$,
a bounded function $b:[0,1]\to\mathbb{R}$, and a real number $c$ such that $a(1)=0$,
\[
f(p)=a(p)+c(p^{\alpha}-p^{\beta})+b(p) \quad \text{if} \quad p\in [0,1], \,\, \beta\neq \alpha
\]
and
\[
f(p)=a(p)+p^{\alpha}\ell(p)+b(p) \quad \text{if} \quad p\in [0,1], \,\, \beta=\alpha\neq 1.
\]
\end{thmszn}

If $\varepsilon=0$ then $b=0$ can be chosen here, so the above theorem is of stability type which however does not cover just the
Shannon case $\beta=\alpha=1$.

In case $\alpha\neq 1$ the problem of the stability of equation \eqref{Eq1.2.2} can easily be handled at least whenever both $n$ and $m$
are not less then three. First of all, introducing a new function $g$ by $g(p)=p+(2^{1-\alpha}-1)f(p), \,\, p\in I$, the stability inequality

\begin{equation*}
\left|\sum_{i=1}^{n}\sum_{j=1}^{m}f(p_{i}q_{j})-\sum_{i=1}^{n}f(p_{i})-\sum_{j=1}^{m}f(q_{j})-(2^{1-\alpha}-1)\sum_{i=1}^{n}f(p_{i})\sum_{j=1}^{m}f(q_{j})\right|\leq\varepsilon
\end{equation*}

goes over into

\begin{equation}\label{Eq3.4.1}
\left|\sum_{i=1}^{n}\sum_{j=1}^{m}g(p_{i}q_{j})-\sum_{i=1}^{n}g(p_{i})\sum_{j=1}^{m}g(q_{j})\right|\leq\varepsilon 
\cdot \left|2^{1-\alpha}-1 \right|
\end{equation}

and the following theorem can be proved. (see Maksa \cite{Mak94})

\begin{thmszn}
Let $n\geq 3, m\geq 3,\,\, 0\leq\varepsilon\in\mathbb{R},\,\, g:[0,1]\to\mathbb{R},\,$ and suppose that \eqref{Eq3.4.1} holds
for all $(p_{1}, \ldots, p_{n})\in\Gamma_{n},\, (q_{1}, \ldots, q_{n})\in\Gamma_{m}$. Then e
\[
g(p)=a(p)+m(p)+b(p) \qquad (p\in [0,1])
\]
where $a:\mathbb{R}\to\mathbb{R}$ is an additive, $b:[0,1]\to\mathbb{R}$ is a bounded, and $m:[0,1]\to\mathbb{R}$ is a multiplicative
function, respectively.
\end{thmszn}

The corner point in the proofs of these theorems is the following stability result (see \cite{Mak94}).
\begin{thmszn}
Let $n\geq 3, \,\, 0\leq\varepsilon\in\mathbb{R},\,\, \varphi:[0,1]\to\mathbb{R},\,$ and suppose that
\begin{equation}\label{Eq3.4.2}
\left|\sum_{i=1}^{n}\varphi(p_{i})\right|\leq\varepsilon
\end{equation}
holds for all $(p_{1}, \ldots, p_{n})\in\Gamma_{n}$. Then there exist an additive function $A:\mathbb{R}\to\mathbb{R}$
and a function $b:[0,1]\to\mathbb{R}$ such that $b(0)=0, \,\,|b(x)|\leq\varepsilon$ for all $x\in [0,1]$ and
\[
\varphi(p)-\varphi(0)=A(p)+b(p) \qquad  (p\in [0,1]).
\]
\end{thmszn}

By an argument similar to that we used in the subsection on sum form equations in connection with
the inequality \eqref{Eq2.2.3}, inequality \eqref{Eq3.4.2} and the triangle inequality imply that
\[
|\varphi(x+y)-\varphi(x)-\varphi(y)+\varphi(0)|\leq 2\varepsilon,
\]
that is, the classical stability inequality holds for the function $\varphi-\varphi(0)$ on the
restricted domain $\{(x,y)\in\mathbb{R}^{2}\, \vert \,  x,y, x+y \in [0,1]\}$. Therefore results
(see Skof \cite{Sko83}, Tabor and Tabor \cite{TabTab98}) on the stability of the Cauchy equation on restricted domain
can be applied to finish the proof of the above theorem.

We remark that the other basic tool for proving stability results for sum form equations was the analysis of the methods
with the help of which the solutions of these equations were found. These and similar ideas proved to be fruitful in the investigations 
of the the stability of the sum form equations on open domain (excluding zero probabilities) and also of the several variable case.
(See Kocsis \cite{Koc00}, \cite{Koc01}, \cite{Koc04}.)


\providecommand{\bysame}{\leavevmode\hbox to3em{\hrulefill}\thinspace}
\providecommand{\MR}{\relax\ifhmode\unskip\space\fi MR }
\providecommand{\MRhref}[2]{%
  \href{http://www.ams.org/mathscinet-getitem?mr=#1}{#2}
}
\providecommand{\href}[2]{#2}

\end{document}